\documentclass[leqno,10pt,a4paper]{article}

\usepackage{a4wide}

\title{Monads for framed torsion-free sheaves on multi-blow-ups of the projective plane }

\author{Amar  Abdelmoubine Henni $^1$ }

\date{}

\usepackage{latexsym}
\usepackage{amsmath}
\usepackage{amsfonts}
\usepackage{amssymb}
\usepackage{amsthm}

\usepackage[T1]{fontenc}

\usepackage[all]{xypic}

\usepackage{psfig}

\usepackage[dvips]{graphicx}

\newcommand{\C}{\mathbb{C}}

\newtheorem{thm}{Theorem}[section]
\newtheorem{pr}[thm]{Proposition}
\newtheorem{cor}[thm]{Corollary}
\newtheorem{lem}[thm]{Lemma}
\newtheorem{definition}[thm]{Definition}
\newtheorem{rmk}[thm]{Remark}

\DeclareMathOperator{\ext}{Ext}
\DeclareMathOperator{\Hom}{Hom}
\DeclareMathOperator{\h}{H}

\DeclareMathOperator{\tor}{Tor}
\DeclareMathOperator{\spec}{Spec }
\DeclareMathOperator{\End}{End}

\begin{document}

\maketitle
\footnotetext[1]{henni@ime.unicamp.br}

\begin{abstract}
We construct monads for framed torsion-free sheaves on blow-ups of the complex projective plane at finitely many distinct points. Using these monads we prove that the moduli space of such sheaves is a smooth algebraic variety. Moreover we construct monads for families of such sheaves parameterized by a noetherian scheme $S$ of finite type. A universal monad on the moduli space is introduced and used to prove that the moduli space is fine.

\end{abstract}

\bigskip\bigskip

\section{Introduction}
In this paper we are concerned with the construction of the moduli space $\mathcal{M}^{\tilde{\mathbb{P}}}_{\vec{a},k}$ of framed torsion-free sheaves of a fixed Chern character on a multi-blow-up of the complex projective plane: $\pi:\tilde{\mathbb{P}}\longrightarrow\mathbb{P}^{2}$, by using monadic descriptions which lead to a parametrization of $\mathcal{M}^{\tilde{\mathbb{P}}}_{\vec{a},k}$ in terms of linear data of the ADHM (Atiyah-Drinfel'd-Hitchin-Manin) type \cite{ADHM}. The ADHM data will be useful, at a first step, to give a presentation of the moduli space as a quotient $\mathcal{M}^{\tilde{\mathbb{P}}}_{\vec{a}, k}=P/G$ where $P$ is a space of some matrices satisfying certain conditions and which will be described below. At a second step the monadic description is used to prove that the space $\mathcal{M}^{\tilde{\mathbb{P}}}_{\vec{a},k}$ is a smooth algebraic variety. This is done by generalizing Buchdahl's construction for holomorphic bundles \cite{Buch}, in order to extend it to torsion-free sheaves.
An additional result is the construction of a monad corresponding to an $S-$flat family $\mathcal{F}$ on a product $\tilde{\mathbb{P}}\times S$, where $S$ is a noetherian scheme of finite type. In particular, there is a universal monad on $\tilde{\mathbb{P}}\times\mathcal{M}^{\tilde{\mathbb{P}}}_{\vec{a}, k}$. Using the ADHM presentation of the moduli space $\mathcal{M}^{\tilde{\mathbb{P}}}_{\vec{a}, k}$ and the properties of the universal monads constructed, we prove that the scheme $\mathcal{M}^{\tilde{\mathbb{P}}}_{\vec{a}, k}$ is a fine moduli space.

Another way of treating the moduli space is to show that one can choose a polarization on $\tilde{\mathbb{P}}$ such that a framed sheaf $(\mathcal{E}, \Phi)$ is stable in the sense of Huybrechts and Lehn (Nakajima \cite{Naka1}, Bruzzo and Markushevich \cite{bruzzo}), and using their result that the moduli space of such objects is a quasi-projective scheme \cite[Theorem 0.1]{Huy1},\cite{Huy2}. Moreover this moduli space is fine \cite[Theorem 0.1]{Huy1} and its smoothness follows from the vanishing of the obstruction in \cite[Section 4]{Huy1}, while in the present work, the smoothness proof is based on the ADHM construction, where the moduli space is a quotient of an affine space by a non-reductive group.

The equivalence between the two approaches is established by the fact that in both cases the moduli space is fine, as we prove for the moduli space $\mathcal{M}^{\tilde{\mathbb{P}}}_{\vec{a}, k}$ in this work. This generalizes the result by Nakajima \cite{Naka} and Okonek et. al. \cite{Okonek} in the cases of Hilbert schemes of points on the projective plane, and rank-2 stable bundles on the projective plane, respectively. This comparison also implies that $\mathcal{M}^{\tilde{\mathbb{P}}}_{\vec{a}, k}$ is quasi-projective.

In the present paper all the varieties (or schemes) are over the field $k=\C.$ For every given coherent sheaves $\mathcal{F}$ and $\mathcal{G}$ over a variety $X$, we denote by $\mathcal{G}\longrightarrow\Hom_{\mathcal{O}_X}(\mathcal{F},\mathcal{G})$ the functor from the category of coherent sheaves of $\mathcal{O}_{X}$-modules to the category of abelian groups (then $\Hom_{\mathcal{O}_X}(\mathcal{F},\mathcal{G})$ is the group of homomorphisms of sheaves of $\mathcal{O}_{X}$-modules), and by $\mathcal{G}\longrightarrow\ext^{i}_{\mathcal{O}_X}(\mathcal{F},\mathcal{G})$ its $i$-th right derived functor. We also denote by $\mathcal{G}\longrightarrow\mathcal{H}om_{\mathcal{O}_X}(\mathcal{F},\mathcal{G})$ the functor from the category of coherent sheaves of $\mathcal{O}_{X}$-modules to itself  (then $\mathcal{H}om_{\mathcal{O}_X}(\mathcal{F},\mathcal{G})$ is the sheaf of local homomorphism groups of sheaves of $\mathcal{O}_{X}$-modules defined as the sheaf associated to the pre-sheaf $U\longrightarrow\Hom_{\mathcal{O}_{X}|_U}(\mathcal{F}|_{U},\mathcal{G}|_{U})$), and by $\mathcal{G}\longrightarrow\mathcal{E}xt^{i}_{\mathcal{O}_X}(\mathcal{F},\mathcal{G})$ its $i$-th right derived functor. For more details see \cite{Godement,Grothendieck0,Hart}. To avoid overloading the text we omit the subscript $\mathcal{O}_{X}$ except if needed; for example on $\mathbb{P}^{2}$ we just write $\ext^{i}(\mathcal{F},\mathcal{G})$ instead of $\ext^{i}_{\mathcal{O}_{\mathbb{P}^{2}}}(\mathcal{F},\mathcal{G})$ and $\mathcal{E}xt^{i}(\mathcal{F},\mathcal{G})$ instead of $\mathcal{E}xt^{i}_{\mathcal{O}_{\mathbb{P}^{2}}}(\mathcal{F},\mathcal{G}).$

\paragraph{Acknowledgement}
I would like to thank first my supervisor Ugo Bruzzo, to whom I am very grateful. Thanks a lot to Claudio Rava for useful discussions, remarks and suggestions. Many thanks to the Department of Mathematics of the University of Genova for helping me during my brief visits. I would also like to thank Professor Vladimir Rubtsov for his interest in my work and for help while he was visiting SISSA in summer 2008.

\section{The construction of the monad}
Before starting our construction we introduce some definitions about the objects we shall study in this paper and will state some of their properties.

\begin{definition}[Monad]\label{monad}
A \underline{monad} $M$ on a scheme $X$ is a complex $\xymatrix@C-1.2pc{M: &0\ar[r]& \mathcal{U}\ar[r]^A & \mathcal{W}\ar[r]^B & \mathcal{V}\ar[r]&0}$ of vector bundles $\mathcal{U}$, $\mathcal{W}$ and $\mathcal{V}$ on $X$ which is exact except at $\mathcal{W}$; i.e., the bundle map $A$ is injective, the bundle map $B$ is surjective and their composition $B\circ A$ is zero. The vector bundle $\mathcal{E}:=KerB/ImA$ is the cohomology of the monad $M.$
\end{definition}

\begin{definition}[The monad display]
Each monad $M$ comes with an associated commutative diagram of exact sequences: $$\xymatrix@C-1pc@R-1pc{ & 0\ar[d] & 0\ar[d] &  & \\
& \mathcal{U}\ar@{=}[r]\ar[d] & \mathcal{U}\ar[d]^A &  & \\
0\ar[r]& \mathcal{X} \ar[r]\ar[d] & \mathcal{W}\ar[r]^B\ar[d] & \mathcal{V}\ar@{=}[d]\ar[r] & 0\\
0\ar[r]& \mathcal{E}\ar[r]\ar[d] & Q \ar[r]\ar[d] & \mathcal{V}\ar[r] & 0\\
& 0 & 0 &  &
}$$
called the \underline{display} of $M$, where $\mathcal{X}=KerB$ and $Q=CokerA$. The conditions on the injectivity of the map $A$ and the surjectivity of the map $B$ are called the non-degeneracy conditions for $M.$

\end{definition}

\begin{rmk}

\begin{itemize}
\item[(i)] \textnormal{Let $X$ be a $2-$dimensional scheme. Recall that if the morphism $\mathcal{U}\stackrel{A}{\longrightarrow}\mathcal{V}$ is a morphism in the category of bundles, then $kerA$ is a subbundle of $\mathcal{U}$ and $ImA$ is a subbundle of $\mathcal{V}.$
If the first map $A$ is only injective as a sheaf map, i.e., $A(x)$ fails to be injective, as a bundle map, at a finite number of points $x\in X$, then the cohomology $\mathcal{E}:=KerB/ImA$ of the monad is no longer a vector bundle: the sheaves $\mathcal{U}$ and $KerB$ are locally free thus their quotient $\mathcal{E}$ is locally free if and only if $A(x)$ is injective as a bundle map. Moreover the singularity locus of $\mathcal{E}$ is supported exactly on the points on which $A(x)$ fails to be injective. This locus has codimension $2,$ thus $\mathcal{E}$ is torsion-free. The non-degeneracy conditions are reduced to the surjectivity of $B$ and the injectivity of $A$ at all $X$ except at finitely many points.}

\item[(ii)]\textnormal{Allowing for generic surjectivity of the map $B$ will not be considered here; permitting this would introduce a cohomology $\mathcal{J}:=Coker(B)$ in the last term of the monad. Hence the monad in that case would parameterize a couple of sheaves: $(KerB/ImA, \mathcal{J}).$ This can be thought of as a particular case of a perverse sheaf. Examples of this kind of monads on $\mathbb{P}^{3}$ are treated in \cite{HL}.
}
\end{itemize}
\end{rmk}

Now assume that $X$ is a nonsingular complex projective surface.

\begin{definition}[Framing]
We say that a torsion-free sheaf $\mathcal{F}$, of rank $r$, is \underline{framed} on the divisor $D\subset X$ if $\mathcal{F}|_{D}$ is trivial and there is a fixed holomorphic trivialization $\Phi:\mathcal{F}|_{D}\longrightarrow\mathcal{O}|_{D}^{\oplus r}$ which is an isomorphism called the \underline{framing} of the sheaf $\mathcal{F}$ on $D$.
\end{definition}

In what follows, we take $X=\tilde{\mathbb{P}}$ where $\pi:\tilde{\mathbb{P}}\longrightarrow\mathbb{P}^{2}$ is the blow-up of the projective plane at $n$ distinct points. $\tilde{\mathbb{P}}$ is regular ($\h^{1}(\tilde{\mathbb{P}},\mathcal{O})=0$) and its Picard group (which in our case can be identified with $\h^{2}(\tilde{\mathbb{P}},\mathbb{Z})$) is generated by $n+1$ elements, namely: $Pic(\tilde{\mathbb{P}})=\oplus_{i=1}^{n}E_{i}\mathbb{Z}\oplus H\mathbb{Z},$ where every $E_{i}$ is an exceptional divisor with the following intersection numbers: $E_{i}^{2}=-1$, $E_{i}\cdot E_{j}=0$ for $i\neq j$, $E_{i}\cdot H=0$ and where $H$ is the divisor given by the generic line in $\mathbb{P}^{2}$ and satisfying $H^{2}=1.$ The canonical divisor of the surface $\tilde{\mathbb{P}}$ is given by $K_{\tilde{\mathbb{P}}}=-3H+\Sigma_{i=1}^{n}E_{i}.$ The Poincaré duals of these divisors are given by $(h,e_{1}, \cdots e_{n})$ where $<h,H>=1$, $<e_{i},E_{j}>=\delta_{ij}$ and also $e_{i}\cdot e_{j}=-\delta_{ij}.$
In terms of line bundles, a divisor of the form $D=pH+\Sigma_{i=1}^{n}q_{i}E_{i}$ has the associated line bundle $\mathcal{O}(D)=\mathcal{O}(p,\overrightarrow{q})=\mathcal{O}(pH)\otimes\mathcal{O}(q_{1}E_{1})\otimes\cdots\otimes\mathcal{O}(q_{n}E_{n})$ where $\overrightarrow{q}=(q_{1},\cdots , q_{n}).$ Then the canonical bundle is given by $\mathcal{O}(K_{\tilde{\mathbb{P}}})=\mathcal{O}(-3H+\Sigma_{i=1}^{n}E_{i})=\mathcal{O}(-3,\overrightarrow{1}).$

The Riemann-Roch formula for a line bundle $\mathcal{O}(p,\overrightarrow{q})$ is given by:
$$\chi(\mathcal{O}(p,\overrightarrow{q}))=\frac{1}{2}[(p+1)(p+2)-|\overrightarrow{q}|^{2}+\Sigma_{i=1}^{n}q_{i}].$$
where $|\overrightarrow{q}|^{2}=\Sigma_{i=1}^{n}q_{i}^{2}.$ We also use the fact that a line bundle $\mathcal{O}(p,\overrightarrow{q})$ restricts to $\mathcal{O}(p)$ on the divisors in the linear system $|\mathcal{O}(H)|$ and restricts to $\mathcal{O}(-q_{i})$ on the divisors in the linear system $|\mathcal{O}(E_{i})|.$

Let $\omega\in \h^{4}(\tilde{\mathbb{P}},\mathbb{Z})$ be the fundamental class of $\tilde{\mathbb{P}}$. For a torsion-free sheaf $\mathcal{E}$, of Chern character $ch(\mathcal{E})=r+(aH+\Sigma_{i=1}^{n}a_{i}E_{i})-(k-\frac{a^{2}-|\overrightarrow{a}|^{2}}{2})\omega$, twisted by a line bundle $\mathcal{O}(p,\overrightarrow{q})$ the Riemann-Roch formula is given by:
$$\chi(\mathcal{E}(p,\overrightarrow{q}))=-[k-\frac{a}{2}(a+3)+\frac{1}{2}\Sigma_{i=1}^{n}a_{i}(a_{i}-1)]+
\frac{r}{2}[(p+1)(p+2)-\Sigma_{i=1}^{n}q_{i}(q_{i}-1)]+[ap-\Sigma_{i=1}^{n}a_{i}q_{i}].$$

\bigskip
General results about surfaces and their blow-ups can be found in details in \cite{beauville, barth}, so that one can reconstruct the above description. We remind the reader that, in \cite{Buch}, Buchdahl has a different notation for line bundles; he uses a notation involving the Poincaré duals $h,e_{1}, \cdots e_{n}$ of the divisors $H,E_{1}, \cdots E_{n}$ respectively.

\bigskip

Now we restrict ourselves to the case of torsion-free sheaves $\mathcal{E}$ with Chern character $ch(\mathcal{E})=r+(\Sigma_{i=1}^{n}a_{i}E_{i})-(k+\frac{|\overrightarrow{a}|^{2}}{2})\omega$ which are framed on a fixed rational curve $l_{\infty}$ in the linear system $|H|,$ i.e., we have a fixed trivialization $\Phi:\mathcal{E}|_{l_{\infty}}\longrightarrow\mathcal{O}|_{l_{\infty}}^{\oplus r}$. The direct image $\pi_{\ast}\mathcal{E}$ of $\mathcal{E}$ is a normalized torsion-free sheaf since it is framed on $\mathbb{P}^{2},$ so that $c_{1}(\pi_{\ast}\mathcal{E})=0$ (a rank $r$ sheaf $\mathcal{F}$ on $\mathbb{P}^{2}$ is normalized if its first Chern class $c_{1}(\mathcal{F})$ satisfies $|c_{1}(\mathcal{F})|<r$).

From the natural injection of the sheaf $\mathcal{E}$ in its double dual $\mathcal{E}^{\ast\ast}$, we have the following exact sequence:
\begin{equation}\label{doubledual}
0\longrightarrow\mathcal{E}\longrightarrow\mathcal{E}^{\ast\ast}\longrightarrow\Delta\longrightarrow0
\end{equation}
where $\mathcal{E}^{\ast\ast}$ has Chern character $ch(\mathcal{E}^{\ast\ast})=r+(\Sigma_{i=1}^{n}a_{i}E_{i})-(k-l+\frac{|\overrightarrow{a}|^{2}}{2})\omega$, and $l$ is the length of the quotient sheaf $\Delta$ supported on finitely many points with  $Supp(\Delta)\cap l_{\infty}=\emptyset$.

\begin{pr}
$\h^{0}(\tilde{\mathbb{P}}, \mathcal{E}^{\ast\ast}(p,\overrightarrow{q}))=\h^{0}(\tilde{\mathbb{P}}, \mathcal{E}^{\ast}(p,\overrightarrow{q}))=0$ $\qquad\forall \overrightarrow{q}$ if $p<0$ \quad and

\hspace{2.5cm}$\h^{2}(\tilde{\mathbb{P}}, \mathcal{E}^{\ast\ast}(p,\overrightarrow{q}))=\h^{2}(\tilde{\mathbb{P}}, \mathcal{E}^{\ast}(p,\overrightarrow{q}))=0$ $\qquad\forall \overrightarrow{q}$ if $p=-1,-2.$
\end{pr}
\begin{proof}
The first vanishing follows by taking the direct image $\pi_{\ast}(\mathcal{E}^{\ast\ast})$ of $\mathcal{E}^{\ast\ast}$. By using the framing condition one can easily verify that, on $\mathbb{P}^{2}$, the group $\h^{0}(\mathbb{P}^{2}, \pi_{\ast}(\mathcal{E}^{\ast\ast})(p))$ vanishes for $p<0.$ But $\h^{0}(\mathbb{P}^{2}, \pi_{\ast}(\mathcal{E}^{\ast\ast})(p))\cong\h^{0}(\tilde{\mathbb{P}}, \mathcal{E}^{\ast\ast}(p,\overrightarrow{q})),$ thus the latter is zero for $p<0$. This also holds for the dual $\mathcal{E}^{\ast}$ since it is also framed. Finally by Serre duality the last two conditions follow easily.
\end{proof}

\begin{cor}\label{cor}
$\h^{0}(\tilde{\mathbb{P}}, \mathcal{E}(p,\overrightarrow{q}))=0$ $\qquad\forall \overrightarrow{q}$ if $p<0$ \quad and

\hspace{2cm} $\h^{2}(\tilde{\mathbb{P}}, \mathcal{E}(p,\overrightarrow{q}))=0$ $\qquad\forall \overrightarrow{q}$ if $p=-1,-2.$
\end{cor}

We also use the following form of Serre-Grothendieck duality for coherent sheaves \cite{Grothendieck0,Hart1}:
\begin{thm}
On a smooth algebraic projective variety $X$ of dimension $n$ over an algebraically closed field $k$, and for every two coherent sheaves $\mathcal{F}$ and $\mathcal{J}$, the following formula holds:
$$\ext^{i}(\mathcal{F},\mathcal{J})=\ext^{n-i}(\mathcal{J},\mathcal{F}\otimes\omega_{X})^{\ast}$$
where $\omega_{X}$ is the canonical sheaf.
\end{thm}

\begin{definition}[Non-locally free monad]
A \underline{non-locally free monad} is a complex $$\xymatrix@C-1.2pc{M: &0\ar[r]& \mathcal{U}\ar[r]^A & \mathcal{W}\ar[r]^B & \mathcal{V}\ar[r]&0}$$ as in definition \ref{monad} in which $\mathcal{U}$, $\mathcal{V}$ and $\mathcal{W}$ are coherent sheaves (not necessarily locally free sheaves).
\end{definition}

We start our program by showing that a framed torsion-free sheaf $\mathcal{E}$ can be described as the cohomology of a non-locally free monad with a torsion-free sheaf in its middle term, and locally free sheaves in the first and the third terms. Since the maps involved are sheaf maps, instead of being bundle maps, one cannot obtain an ADHM description. This problem will be solved later by constructing a monad, out of the obtained non-locally free monad, and proving that its cohomology is the starting sheaf $\mathcal{E}.$ This construction is mainly a generalization of \cite[Section 1]{Buch} to the case of framed torsion-free sheaves.

First define the spaces $B_{i}:=\Hom(\mathcal{E},\mathcal{O}|_{E_{i}}(-1))^{\ast}$ and let $\mathcal{B}_{1}=\oplus_{i=1}^{n} B_{i}(1,-E_{i})$. Then the extensions of the form $$0\longrightarrow\mathcal{E}\longrightarrow Q_{1}\longrightarrow\mathcal{B}_{1}\longrightarrow0$$
are classified by the group $\ext^{1}(\mathcal{B}_{1},\mathcal{E})\cong\oplus_{i=1}^{n} B_{i}^{\ast}\otimes \ext^{1}(\mathcal{E}, \mathcal{O}(-2,\overrightarrow{1}-E_{i}))^{\ast}$.
Applying the functor $\Hom(\mathcal{E},\cdot)$ to the sequence
\begin{equation}\label{rest2}
0\longrightarrow\mathcal{O}(0,-E_{i})\longrightarrow\mathcal{O}\longrightarrow\mathcal{O}|_{E_{i}}\longrightarrow0
\end{equation}
after twisting by $\mathcal{O}(-2,\overrightarrow{1})$ one obtains $$\xymatrix{\Hom(\mathcal{E},\mathcal{O}(-2,\overrightarrow{1}))\ar[r]& \Hom(\mathcal{E},\mathcal{O}|_{E_{i}}(-1))\ar[r]& \ext^{1}(\mathcal{E},\mathcal{O}(-2,\overrightarrow{1}-E_{i}))}$$
but $\Hom(\mathcal{E},\mathcal{O}(-2,\overrightarrow{1}))=\h^{2}(\tilde{\mathbb{P}}, \mathcal{E}(-1,0))^{\ast}=0$ by the corollary above. Then the map
$$\xymatrix{B_{i}^{\ast}\ar@{^{(}->}[r]& \ext^{1}(\mathcal{E},\mathcal{O}(-2,\overrightarrow{1}-E_{i}))}$$
is injective. This implies that the map
$$\xymatrix{\ext^{1}(\mathcal{E},\mathcal{O}(-2,\overrightarrow{1}-E_{i}))^{\ast}\ar[r]^{\qquad\qquad r}& B_{i}}$$
is surjective. Thus there exists an extension $Q_{1}$ which is mapped to the identity in $\End(B_{i})$ for each $i$ under the composition of the projection on the $i-$th factor and the map $r$ above.

Let us now define $A_{i}:=\ext^{1}(\mathcal{E},\mathcal{O}|_{E_{i}}(-1))^{\ast}$ and $\mathcal{A}_{1}:=\oplus_{i=1}^{n}A_{i}(-1,E_{i}).$ The extensions of the form $$0\longrightarrow\mathcal{A}_{1}\longrightarrow X_{1}\longrightarrow\mathcal{E}\longrightarrow0$$
are classified by $\ext^{1}(\mathcal{E},\mathcal{A}_{1})\cong\oplus_{i=1}^{n} A_{i}\otimes \ext^{1}(\mathcal{E},\mathcal{O}(-1,E_{i})).$ Applying the functor $\Hom(\mathcal{E}, \cdot)$ to \eqref{rest2} twisted by $\mathcal{O}(-1,E_{i})$ one has the following exact sequence:
\begin{equation}\label{Adual}
\ext^{1}(\mathcal{E},\mathcal{O}(-1,E_{i}))\longrightarrow \underbrace{\ext^{1}(\mathcal{E},\mathcal{O}|_{E_{i}}(-1))}_{A_{i}^{\ast}}\longrightarrow \ext^{2}(\mathcal{E},\mathcal{O}(-1,0))
\end{equation}
where $\ext^{2}(\mathcal{E},\mathcal{O}(-1,0))=\Hom(\mathcal{O}(-1,0),\mathcal{E}(-3,\overrightarrow{1}))^{\ast} =\h^{0}(\tilde{\mathbb{P}},\mathcal{E}(-2,\overrightarrow{1}))^{\ast}=0$ also by the corollary above. Thus there exists an extension in $\ext^{1}(\mathcal{E},\mathcal{A}_{1})$ which maps to the identity in $\End(A_{i})$ for every $i=1,...,n.$

To construct a display of a non-locally free monad one may apply the following:
\begin{pr}\cite[Proposition \textbf{2.2.3}]{King}\label{King223}
Suppose we are given two extensions
\begin{align}&\xymatrix@C-1pc{0\ar[r] & \mathcal{U}\ar[r]^{i_{1}} & \mathcal{X}\ar[r]^{j_{1}} & \mathcal{E}\ar[r] & 0} \label{ext1}\\
&\xymatrix@C-1pc{0\ar[r] & \mathcal{E}\ar[r]^{i_{4}} &  \mathcal{Q}\ar[r]^{j_{4}} & \mathcal{V}\ar[r] & 0} \label{ext2}.
\end{align} of torsion free sheaves. Then we can fit them into a non-locally free monad display if and only if the double extension
\begin{equation}\label{doublext}
\xymatrix@C-1pc{0\ar[r] & \mathcal{U}\ar[r] & \mathcal{X}\ar[r] &  \mathcal{Q}\ar[r] & \mathcal{V}\ar[r] & 0},
\end{equation}
i.e., their $\ext$-product in $\ext^{2}(\mathcal{V},\mathcal{U}),$ is trivial. Furthermore any two way of completing the display differ by an action of $\ext^{1}(\mathcal{V},\mathcal{U}).$
\end{pr}
\begin{proof}
By applying the contravariant functor $\Hom(\bullet,\mathcal{U})$ on the second sequence one gets $$\xymatrix@C-1pc{\ext^{1}(\mathcal{V},\mathcal{U})\ar[r]^{\tilde{j_{4}}} & \ext^{1}(\mathcal{Q},\mathcal{U})\ar[r]^{\tilde{i_{4}}} & \ext^{1}(\mathcal{E},\mathcal{U}) \ar[r]^{\delta} & \ext^{2}(\mathcal{V},\mathcal{U})\ar[r] & \cdots}.$$ The map $\delta$ sends the extension \eqref{ext1} to the extension \eqref{doublext} in $\ext^{2}(\mathcal{V},\mathcal{U}).$ Conversely when the double extension \eqref{doublext} is trivial one can find an extension
$$\xymatrix@C-1pc{0\ar[r] & \mathcal{U}\ar[r]^{\alpha} &  \mathcal{W}\ar[r]^{j_{2}} & \mathcal{Q}\ar[r] & 0}$$ in $\ext^{1}(\mathcal{Q},\mathcal{U})$ which is mapped to \eqref{ext1} by $\tilde{i_{4}}.$ This implies that there is a uniquely determined map $i_{3}:\mathcal{X}\longrightarrow\mathcal{W}$ such that $$\xymatrix@C-1pc@R-1pc{0\ar[r]& \mathcal{U}\ar[r]\ar@{=}[d]&\mathcal{X}\ar[r]\ar[d]^{i_{3}}&\mathcal{E}\ar[r]\ar[d]^{i_{4}}&0 \\ 0\ar[r]&\mathcal{U}\ar[r]&\mathcal{W}\ar[r]&\mathcal{Q}\ar[r]&0}$$ commutes. Finally, putting $\beta=j_{4}\circ j_{2}$ one has the sequence
$$\xymatrix@C-1pc{0\ar[r] & \mathcal{X}\ar[r]^{i_{3}} &  \mathcal{W}\ar[r]^{\beta} & \mathcal{V}\ar[r] & 0}$$
required to complete the display.

One can verify that $\ext^{1}(\mathcal{V},\mathcal{U})$ acts naturally on the space of all monads with ends $\mathcal{U}$ and $\mathcal{V}.$ The monad we obtained in this way belongs to the orbit of this action.

\end{proof}

In our case one has $\ext^{2}(\mathcal{B}_{1},\mathcal{A}_{1})\cong\oplus_{i,j}^{n}A_{i}\otimes B^{\ast}_{j}\otimes \h^{2}(\tilde{\mathbb{P}},\mathcal{O}(-2,E_{i}+E_{j}))=0$, thus there exists a sheaf $W_{1}$ and exact sequences
$$0\longrightarrow\mathcal{A}_{1}\longrightarrow W_{1}\longrightarrow Q_{1}\longrightarrow0$$
$$0\longrightarrow X_{1}\longrightarrow W_{1}\longrightarrow\mathcal{B}_{1}\longrightarrow0$$
which fit into the following commutative diagram:
\begin{equation}\label{display}
\xymatrix@R-1pc@C-1pc{         &     0       \ar[d]           &      0      \ar[d]        &                               &         \\
           & \mathcal{A}_{1} \ar@{=}[r]\ar[d] & \mathcal{A}_{1}  \ar[d] &                 &         \\
          0 \ar[r] &     X_{1}       \ar[r]\ar[d]     & W_{1} \ar[r] \ar[d] & \mathcal{B}_{1} \ar[r] \ar@{=}[d] &    0    \\
          0 \ar[r] & \mathcal{E} \ar[r]\ar[d]     &      Q_{1}     \ar[r] \ar[d] & \mathcal{B}_{1} \ar[r]      & 0 \\
                   &     0                        &      0                    &    &       \\
}
\end{equation}
and thus one has a non-locally free monad
\begin{equation}\label{M1}
M_{1}:\quad 0\longrightarrow\mathcal{A}_{1}\longrightarrow W_{1}\longrightarrow\mathcal{B}_{1}\longrightarrow0
\end{equation}
with cohomology the torsion-free sheaf $\mathcal{E}.$ Note that when the extension $Q_{1}$ was constructed, we also required that the induced maps $B_{i}\longrightarrow B_{i},$ in cohomology, are all isomorphisms.

For further computations one needs to know the Chern characters of the sheaves involved in the display. By the Riemann-Roch theorem one, first, has $\chi(\mathcal{E},\mathcal{O}|_{E_{i}}(-1))=a_{i}$, where the Euler characteristic of a pair of coherent sheaves $(\mathcal{F}, \mathcal{G})$ on an algebraic variety $X$ is $\chi(\mathcal{F},\mathcal{G}):=\Sigma_{i=0}^{dimX}(-1)^{i}dim\ext^{i}(\mathcal{F},\mathcal{G})$, see \cite[Definition \textbf{6.1.1}]{Huy}. If we put $d_{i}=dim\ext^{1}(\mathcal{E}, \mathcal{O}|_{E_{i}}(-1))=dimA_{i}$ and $d'_{i}=dim\Hom(\mathcal{E}, \mathcal{O}|_{E_{i}}(-1))=dimB_{i}$ then $d_{i}-d'_{i}=-a_{i}$. We also put $D=\Sigma_{i=1}^{n}d_{i}=rk\mathcal{A}_{1}$ and $D'=rk\mathcal{B}_{1}=\Sigma_{i=1}^{n}d'_{i}$, then $D-D'=-\Sigma_{i=1}^{n}a_{i}:=-\bar{a}$.  It follows that:
$$ch(\mathcal{A}_{1})=D-[DH-\Sigma_{i=1}^{n}d_{i}E_{i}]$$
$$ch(\mathcal{B}_{1})=D'+[D'H-\Sigma_{i=1}^{n}d'_{i}E_{i}]$$

By the additivity of the Chern character on exact sequences one has the following:
$$ch(X_{1})=ch(\mathcal{A}_{1})+ch(\mathcal{E})=(r+D)-[DH-\Sigma_{i=1}^{n}(d_{i}+a_{i})E_{i}]-(k+\frac{|\overrightarrow{a}|^{2}}{2})\omega$$
$$ch(Q_{1})=ch(\mathcal{B}_{1})+ch(\mathcal{E})=(r+D')+[D'H-\Sigma_{i=1}^{n}(d'_{i}-a_{i})E_{i}]-(k+\frac{|\overrightarrow{a}|^{2}}{2})\omega$$
$$ch(W_{1})=ch(\mathcal{A}_{1})+ch(Q_{1})=ch(\mathcal{B}_{1})+ch(X_{1})=(r+D+D')+\bar{a}H-(k+\frac{|\overrightarrow{a}|^{2}}{2})\omega$$

From the lower row of the display \eqref{display}, it follows that $Q_{1}$ is torsion-free since $\mathcal{E}$ is torsion-free and $\mathcal{B}_{1}$ is locally free. Similarly from the middle column of the display it follows that $W_{1}$ is also torsion-free. This implies from the middle row that  $X_{1}$ is a torsion-free sheaf. Furthermore, by dualizing twice the exact sequences in the display one can show that $\Delta_{X}\cong\Delta_{W}$ and $\Delta_{Q}\cong\Delta,$ where we define $\Delta_{W}:=W_{1}^{\ast\ast}/W_{1}$, $\Delta_{X}:=X_{1}^{\ast\ast}/X_{1}$ and $\Delta_{Q}:=Q_{1}^{\ast\ast}/Q_{1}.$ So the complex \eqref{M1} has the disadvantage of containing a non-locally free sheaf in its middle term. To solve this problem we need to construct a monad according to the definition \textbf{\ref{monad}},  i.e., with holomorphic bundles in all its terms. This will be done in few steps, but first we need the following:
\begin{pr}\label{push&pull}
Let $\mathcal{F}$ be a torsion-free sheaf on $\tilde{\mathbb{P}}$, satisfying $$\h^{0}(E_{i}, \mathcal{F}|_{E_{i}}(-1))=0,$$
$$\h^{1}(E_{i}, \mathcal{F}|_{E_{i}}(-1))=0,$$  for some $i$, and let $\Delta_{\mathcal{F}}$ be the quotient sheaf $(\mathcal{F}^{\ast\ast}/\mathcal{F})$. Then $Supp(\Delta_{\mathcal{F}})\cap E_{i}=\emptyset$ and $\mathcal{F}$ is trivial on the divisor $E_{i}$. Moreover if this holds for every $i$, then $\mathcal{F}\cong \pi^{\ast}(\pi_{\ast}\mathcal{F})$.
\end{pr}

\begin{proof}

Let us consider the sequence $0\longrightarrow\mathcal{F}\longrightarrow\mathcal{F}^{\ast\ast}\longrightarrow\Delta_{\mathcal{F}}\longrightarrow0.$ Restricting to $E_{i}$ we have:
$$\xymatrix@R-1.5pc@C-1pc{0\ar[r]&T\ar[r]&\mathcal{F}|_{E_{i}}\ar[rr]\ar@{-->}[rd]&&\mathcal{F}^{\ast\ast}|_{E_{i}}\ar[r] &\Delta_{\mathcal{F}}\otimes\mathcal{O}|_{E_{i}}\ar[r]&0\\
&&&\mathcal{L}\ar@{-->}[rd]\ar@{-->}[ru]&&&\\
&&0\ar@{-->}[ru]&&0&&
}$$ where $T:=\tor^{1}(\mathcal{O}|_{E_{i}}, \Delta_{\mathcal{F}})$ and $\mathcal{L}$ is locally free. Since $E_{i}$ is a curve, then $\mathcal{F}|_{E_{i}}\cong\mathcal{L}\oplus T.$ The first condition means that $\h^{0}(E_{i}, \mathcal{L}(-1))\oplus \h^{0}(E_{i}, T(-1))=0$, thus $T=\tor^{1}(\mathcal{O}|_{E_{i}}, \Delta_{\mathcal{F}})=0$. It follows that $Supp(\Delta_{\mathcal{F}})\cap E_{i}=\emptyset$. Then $\Delta_{\mathcal{F}}\otimes\mathcal{O}|_{E_{i}}=0$, and the sheaf $\mathcal{F}|_{E_{i}}$ is isomorphic to $\mathcal{F}^{\ast\ast}|_{E_{i}}.$ The second condition means that $\mathcal{F}|_{E_{i}}$ is trivial on $E_{i}$.
Now if the conditions hold for every divisor $E_{i}$, then $\mathcal{F}|_{E_{i}}$ is trivial on every $E_{i}$ and $Supp(\Delta_{\mathcal{F}})\cap E_{i}=\emptyset\quad\forall i$. Hence $\mathcal{F}$ is the pull-back of its direct image on $\mathbb{P}^{2}$.

\end{proof}

\begin{lem}
$W_{1}\cong\pi^{\ast}(\pi_{\ast}W_{1})$.
\end{lem}
\begin{proof}

Twisting the monad by $\mathcal{O}(E_{i})$ and restricting to the exceptional divisor $E_{i}$ one has, from the middle column of the display:
$$\xymatrix@C-1.3pc{0\ar[r]&\tor^{1}(\mathcal{E}(-1), \mathcal{O}|_{E_{i}})\ar[r]&\tor^{1}(X_{1}(-1), \mathcal{O}|_{E_{i}})\ar[r]&\oplus_{j\neq i}A_{j}(-1)\oplus A_{i}(-2)}$$ $$\xymatrix@C-1.2pc{\ar[r]&W_{1}|_{E_{i}}(-1)\ar[r]&Q_{1}|_{E_{i}}(-1)\ar[r]&0}$$ which can be split as

$$\xymatrix@R-1.6pc@C-1.4pc{0\ar[r]&\tor^{1}(\mathcal{E}(-1), \mathcal{O}|_{E_{i}})\ar[r] & \tor^{1}(X_{1}(-1), \mathcal{O}|_{E_{i}})\ar[r]&\mathcal{F}\ar[r]&0 \\
0\ar[r]&\mathcal{F}\ar[r]&\oplus_{j\neq i}A_{j}(-1)\oplus A_{i}(-2)\ar[r]&\mathcal{G}\ar[r]&0 \\
0\ar[r]&\mathcal{G}\ar[r]&W_{1}|_{E_{i}}(-1)\ar[r]&Q_{1}|_{E_{i}}(-1)\ar[r]&0
}$$

The $\tor-$sheaves are supported on points lying on the curve $E_{i}$, and so is the sheaf $\mathcal{F}.$ Thus $\h^{1}(E_{i},\mathcal{F})=0.$
On the other hand $\h^{0}(E_{i},\oplus_{j\neq i}A_{j}(-1)\oplus A_{i}(-2))=0$ implying that $\h^{0}(E_{i},\mathcal{F})=0,$ then $\mathcal{F}$ is the zero sheaf. This means that we have an isomorphism of sheaves $\tor^{1}(\mathcal{E}(-1) \mathcal{O}|_{E_{i}})\cong\tor^{1}(X_{1}(-1), \mathcal{O}|_{E_{i}}).$ Moreover the sequence $$\xymatrix@C-1.4pc{0\ar[r]&\oplus_{j\neq i}A_{j}(-1)\oplus A_{i}(-2)\ar[r]&W_{1}|_{E_{i}}(-1)\ar[r]&Q_{1}|_{E_{i}}(-1)\ar[r]&0}$$ is exact. Its long exact sequence in cohomology gives:
\begin{align}\label{trivW}
0&\longrightarrow\underbrace{\h^{0}(E_{i},\oplus_{j\neq i}A_{j}(-1)\oplus A_{i}(-2))}_{0}\longrightarrow \h^{0}(E_{i},W_{1}|_{E_{i}}(-1))\longrightarrow \h^{0}(E_{i},Q_{1}|_{E_{i}}(-1)) \notag \\
& \\
&\longrightarrow\underbrace{\h^{1}(E_{i},\oplus_{j\neq i}A_{j}(-1)\oplus A_{i}(-2))}_{A_{i}}\longrightarrow \h^{1}(E_{i},W_{1}|_{E_{i}}(-1))\longrightarrow \h^{1}(E_{i},Q_{1}|_{E_{i}}(-1))\longrightarrow0 \notag
\end{align}
but from the last row of the display, i.e. $\xymatrix@C-1.3pc{0 \ar[r] & \mathcal{E}|_{E_{i}}(-1) \ar[r]&Q_{1}|_{E_{i}}(-1)\ar[r] & \oplus_{j\neq i}B_{j}(-1)\oplus B_{i} \ar[r] & 0}$, one has
$$\xymatrix{0\ar[r]& A_{i}\ar[r]& \h^{0}(E_{i},Q_{1}|_{E_{i}}(-1))\ar[r]& B_{i}\ar[r]^\sim & B_{i}\ar[r]& \h^{1}(E_{i},Q_{1}|_{E_{i}}(-1))\ar[r]&0}$$  which means that $$\h^{0}(E_{i},Q_{1}|_{E_{i}}(-1))\cong A_{i},\quad\quad \h^{1}(E_{i},Q_{1}|_{E_{i}}(-1))=0$$ and $$\quad\quad\quad \h^{0}(E_{i},W_{1}|_{E_{i}}(-1))=0, \quad\quad \h^{1}(E_{i},W_{1}|_{E_{i}}(-1))=0 \quad \forall i=1,n.$$

Thus the lemma follows from proposition \textbf{\ref{push&pull}}. We used the fact that the complex $M_{1}$ was constructed so that all the induced maps $B_{i}=\h^{0}(E_{i},\mathcal{B}_{1}\otimes\mathcal{O}|_{E_{i}}(-1))\longrightarrow B_{i}=\h^{0}(E_{i},\mathcal{E}\otimes\mathcal{O}|_{E_{i}}(-1))$ are isomorphisms for each $i.$
\end{proof}

\bigskip

The next step will be the construction of a monad on $\mathbb{P}^{2}$ which describes the sheaf $\pi_{\ast}W_{1}$. For this we need the following:
\begin{thm}\label{thm1}
A torsion-free sheaf $\mathcal{F}$ on $\mathbb{P}^{2}$ is given by the cohomology of a monad with trivial middle term if
$$\h^{0}(\mathbb{P}^{2},\mathcal{F}(-1))=0, \qquad \textrm{and}\qquad \h^{0}(\mathbb{P}^{2},\mathcal{F}^{\ast}(-1))=0$$
\end{thm}

\begin{proof}
Beilinson's theorem (\cite{Okonek} \textbf{3.1.3} and \textbf{3.1.4}) extends to the case of a torsion-free sheaf (\cite{Naka} \textbf{2.1}, \cite{Ancona1} and \cite{Ancona} for applications), hence there exists a spectral sequence $E_{r}^{p,q}$ with first term: $E_{1}^{p,q}=\h^{q}(\mathbb{P}^{2},\mathcal{F}\otimes \Omega^{-p}(-p))\otimes\mathcal{O}(p)$
which converges to :
$$E_{\infty}^{p,q}=\left\{\begin{array}{ll}\mathcal{F} & \textrm{for }p+q=0\\0& \textrm{otherwise} \end{array}\right.$$
We apply this to the sheaf $\mathcal{F}(-1)$ and use the vanishing conditions. This leads to a monad, with cohomology $\mathcal{F}(-1)$, given by
$$\xymatrix{0\ar[r]&E_{1}^{-2,1}\ar[r]^{d_{1}^{-2,1}}&E_{1}^{-1,1}\ar[r]^{d_{1}^{-1,1}}&E_{1}^{0,1}\ar[r]&0}$$
Twisting the complex by $\mathcal{O}(-1)$ one has the monad:
$$\xymatrix@C-0.3pc@R-1.4pc{0\ar[r]&\h^{1}(\mathbb{P}^{2},\mathcal{F}(-2))\otimes\mathcal{O}(-1)\ar[r]^{\quad d_{1}^{-2,1}}&\h^{1}(\mathbb{P}^{2},\mathcal{F}\otimes\Omega^{1})\otimes\mathcal{O} &&\\
&\qquad\quad\ar[r]^{d_{1}^{-1,1}\quad}&\h^{1}(\mathbb{P}^{2},\mathcal{F}(-1))\otimes\mathcal{O}(1)\ar[r]&0&&}$$ with cohomology the sheaf $\mathcal{F}.$

\end{proof}

\begin{lem}\label{dualdirectimage}
$\h^{0}(\mathbb{P}^{2},\pi_{\ast}(W_{1}^{\ast})(-1))=\h^{0}(\mathbb{P}^{2},(\pi_{\ast}W_{1})^{\ast}(-1))$
\end{lem}
\begin{proof}
First one has $\h^{0}(\mathbb{P}^{2},\pi_{\ast}(W_{1}^{\ast})(-1))= \h^{0}(\tilde{\mathbb{P}},W_{1}^{\ast}(-1,0))$ which by Serre duality is equal to $\h^{2}(\tilde{\mathbb{P}},W_{1}^{\ast\ast}(-2,\vec{1}))^{\ast}.$ From the natural injection of a torsion-free sheaf in its double dual, one has
\begin{align}
\h^{2}(\tilde{\mathbb{P}},W_{1}^{\ast\ast}(-2,\vec{1}))^{\ast}&=\h^{2}(\tilde{\mathbb{P}},W_{1}(-2,\vec{1}))^{\ast}\notag \\
&=\ext^{2}_{\mathcal{O}_{\tilde{\mathbb{P}}}}(\mathcal{O}(-2,\vec{1}),W_{1})^{\ast} \notag \\
&=\Hom_{\mathcal{O}_{\tilde{\mathbb{P}}}}(W_{1},\mathcal{O}(-1,0)). \notag
\end{align}
Moreover for any sheaf of $\mathcal{O}_{\tilde{\mathbb{P}}}$-modules $\mathcal{F}$ and for any sheaf of $\mathcal{O}_{\mathbb{P}^{2}}$-modules $\mathcal{G}$ one has the formula (\cite{Hart},II. 5 page 110): $\Hom_{\tilde{\mathbb{P}}}(\pi^{\ast}\mathcal{G},\mathcal{F})=\Hom_{\mathbb{P}^{2}}(\mathcal{G},\pi_{\ast}\mathcal{F})$ since $\pi_{\ast}$ and $\pi^{\ast}$ are adjoint functors. Then using the fact that $W_{1}$ is the pull-back of its direct image on $\mathbb{P}^{2}$, and the fact that $\pi_{\ast}\mathcal{O}(-1,0)\cong\mathcal{O}(-1)$ we have the canonical isomorphisms
\begin{align}
\Hom_{\mathcal{O}_{\tilde{\mathbb{P}}}}(\pi^{\ast}(\pi_{\ast}W_{1}),\mathcal{O}(-1,0))&= \Hom_{\mathcal{O}_{\mathbb{P}^{2}}}(\pi_{\ast}W_{1},\mathcal{O}(-1))\notag \\
&=\ext^{2}_{\mathcal{O}_{\mathbb{P}^{2}}}(\mathcal{O}(-1),\pi_{\ast}W_{1}(-3))^{\ast}\notag \\
&=\h^{2}(\mathbb{P}^{2},\pi_{\ast}W_{1}(-2))^{\ast}. \notag
\end{align}
Again, by the natural injection of a torsion-free sheaf in its double dual, one has
\par\noindent
$\h^{2}(\mathbb{P}^{2},\pi_{\ast}W_{1}(-2))^{\ast}= \h^{2}(\mathbb{P}^{2},(\pi_{\ast}W_{1})^{\ast\ast}(-2))^{\ast}= \h^{0}(\mathbb{P}^{2},(\pi_{\ast}W_{1})^{\ast}(-1))$ from which the claim follows.
\end{proof}

\begin{pr}
The direct image $\pi_{\ast}W_{1}$ of the sheaf $W_{1}$ is given by the cohomology of a monad on $\mathbb{P}^{2}$ with trivial middle term.
\end{pr}

\begin{proof}
It suffices to verify the vanishing given in theorem \textbf{\ref{thm1}}.

\bigskip

\underline{$\h^{0}(\mathbb{P}^{2},(\pi_{\ast}W_{1})^{\ast}(-1))=0$:}

\bigskip

From lemma \textbf{\ref{dualdirectimage}} we have $\h^{0}(\mathbb{P}^{2},(\pi_{\ast}W_{1})^{\ast}(-1))=\h^{0}(\mathbb{P}^{2},\pi_{\ast}(W_{1}^{\ast})(-1)).$ On the other hand we have $\h^{0}(\mathbb{P}^{2},\pi_{\ast}(W_{1}^{\ast})(-1))=\h^{0}(\tilde{\mathbb{P}},W_{1}^{\ast}(-1,0))$ since $W_{1}^{\ast}(-1,0)$ is the pullback of $\pi_{\ast}(W_{1}^{\ast})(-1)$ under the blow-down map. Then it suffices to show that $\h^{0}(\tilde{\mathbb{P}},W_{1}^{\ast}(-1,0))=0$ in order to prove the first vanishing: dualizing the first row of the display \eqref{display} and twisting the resulting sequence by $\mathcal{O}(-1,0),$ one has the following sequence in cohomology
$$0\longrightarrow\oplus_{i=1}^{n}B^{\ast}_{i}\otimes\underbrace{\h^{0}(\tilde{\mathbb{P}},\mathcal{O}(-2,-E_{i}))}_{0}\longrightarrow \h^{0}(\tilde{\mathbb{P}},W_{1}^{\ast}(-1,0))\longrightarrow \h^{0}(\tilde{\mathbb{P}},X_{1}^{\ast}(-1,0))$$
Taking the dual sequence of the left column of the display \eqref{display} and twisting by $\mathcal{O}(-1,0),$ one has the following induced exact sequence in cohomology $$0\longrightarrow\underbrace{\h^{0}(\tilde{\mathbb{P}},\mathcal{E}^{\ast}(-1,0))}_{0}\longrightarrow \h^{0}(\tilde{\mathbb{P}},X^{\ast}(-1,0))\longrightarrow \h^{0}(\tilde{\mathbb{P}},\mathcal{A}^{\ast}_{1}(-1,0))$$ but $\h^{0}(\tilde{\mathbb{P}},\mathcal{A}_{1}^{\ast}(-1,0))\cong\oplus_{i=1}^{n}A_{i}^{\ast}\otimes \h^{0}(\tilde{\mathbb{P}},\mathcal{O}(0,-E_{i}))$ where the group $\h^{0}(\tilde{\mathbb{P}},\mathcal{O}(0,-E_{i}))$ is zero for any $i$. Thus $\h^{0}(\tilde{\mathbb{P}},X_{1}^{\ast}(-1,0))=0$ which implies that $W^{\ast}_{1}(-1)$ has no global sections.

\bigskip

\underline{$\h^{0}(\mathbb{P}^{2},\pi_{\ast}W_{1}(-1))=0$:} $\quad$ Again it suffices to show that $\h^{0}(\tilde{\mathbb{P}},W_{1}(-1,0))=0.$ If one twists the display \eqref{display} by $\mathcal{O}(-1,0),$ then the vanishing follows by the same argument as above. Hence the torsion-free sheaf $\pi_{\ast}W_{1}$ is described as the cohomology of a monad on $\mathbb{P}^{2}$ with trivial middle term.

\end{proof}
The monad which has cohomology the sheaf $\pi_{\ast}W_{1}$ is given by
\begin{equation*}
    M'_{0}:\quad 0\longrightarrow K_{0}(-1)\longrightarrow W\longrightarrow L_{0}(1)\longrightarrow0
\end{equation*}
where $K_{0}=\h^{1}(\tilde{\mathbb{P}}, W_{1}(-2,\overrightarrow{1}))$, $L_{0}=\h^{1}(\tilde{\mathbb{P}}, W_{1}(-1,0))$ and $W$ is a trivial bundle. To see that the spaces $K_{0}$ and $L_{0}$ yield the Chern character $ch(W_{1})$ we shall compute their dimensions:

\bigskip

\begin{pr}
    $$\h^{0}(\tilde{\mathbb{P}}, W_{1}(-2,\overrightarrow{1}))=\h^{0}(\tilde{\mathbb{P}}, W_{1}(-1,0))=0$$
 and
    $$\h^{2}(\tilde{\mathbb{P}}, W_{1}(-2,\overrightarrow{1}))=\h^{2}(\tilde{\mathbb{P}}, W_{1}(-1,0))=0$$
\end{pr}
\begin{proof}
The proof is given by using the display \eqref{display} twisted by $\mathcal{O}(-1,0)$ and taking the induced long exact sequences in cohomology.
\end{proof}
\begin{cor}
The spaces $K_{0}$ and $L_{0}$ have dimension $k+\frac{|\overrightarrow{a}|^{2}-\bar{a}}{2}$ and $k+\frac{|\overrightarrow{a}|^{2}+\bar{a}}{2},$ respectively.
\end{cor}
\begin{proof}
Using the Riemann-Roch formula we compute the Euler characteristics of $W_{1}(-2,\overrightarrow{1})$ and $W_{1}(-1,0)$. This gives
\begin{equation}
\chi(W_{1}(-2,\overrightarrow{1}))=-(k+\frac{|\overrightarrow{a}|^{2}+\bar{a}}{2})\quad\textrm{and}\quad \chi(W_{1}(-1,0))=-(k+\frac{|\overrightarrow{a}|^{2}-\bar{a}}{2})
\end{equation}
The corollary follows from the vanishing of the groups in the proposition above.
\end{proof}

\bigskip

Now we want to construct an intermediate monad with trivial middle term and with cohomology the original sheaf $\mathcal{E}.$
First we have to pull-back the monad $M'_{0}$ to a monad $M_{0}$ on $\tilde{\mathbb{P}}$:
\begin{equation*}
    M_{0}:\quad 0\longrightarrow K_{0}(-1,0)\longrightarrow W\longrightarrow L_{0}(1,0)\longrightarrow0
\end{equation*}
Indeed the first map in the monad $M'_{0}$ vanishes on the singularity set $Sing(\pi_{\ast}W_{1})$ of the torsion-free sheaf $\pi_{\ast}W_{1}.$ On the other hand $Sing(W_{1})\cap E_{i}=\emptyset$ for all $i.$ Moreover $W_{1}=\pi^{\ast}\pi_{\ast}W_{1}.$ This implies that $Sing(\pi_{\ast}W_{1})\cap p_{i}=\emptyset$ for all $i$, where $p_{i}\in\mathbb{P}^{2}$ are the blow-up points. Now the first map in $M'_{0}$ has maximal rank at $p_{i}$ for all $i$, as well as the second map. Consequently the locus on which the first map in $M_{0}$ is not of maximal rank is zero dimensional. Thus the pull-back $M_{0}$ of $M'_{0}$ is a monad.

Then we should lift the morphism $\mathcal{A}_{1}\longrightarrow W_{1}$ to a morphism $\mathcal{A}_{1}\longrightarrow X'_{0}$ where $X'_{0}=ker(W\longrightarrow L_{0}(1,0))$ i.e.
\begin{equation*}
    \xymatrix@C-0.5pc@R-0.5pc{&X'_{0}\ar[d] \\
    \mathcal{A}_{1}\ar@{-->}[ru]\ar@{^{(}->}[r]&W_{1}
    }
\end{equation*}
so we want a surjective morphism $\Hom(\mathcal{A}_{1},X'_{0})\longrightarrow \Hom(\mathcal{A}_{1},W_{1}).$ The obstruction for such a lifting lies in the group $\ext^{1}(\mathcal{A}_{1},K_{0}(-1,0))$ which is zero since $\ext^{1}(\mathcal{A}_{1},K_{0}(-1,0))\cong\oplus_{i=1}^{n}A^{\ast}_{i}\otimes K_{0}\otimes \h^{1}(\tilde{\mathbb{P}}, \mathcal{O}(0,-E_{i}))=0.$ This means that all the extensions
$$0\longrightarrow K_{0}(-1,0)\longrightarrow \mathcal{A}\longrightarrow \mathcal{A}_{1}\longrightarrow0$$
split, hence $\mathcal{A}\cong K_{0}(-1,0)\oplus\mathcal{A}_{1}.$ Furthermore we have a sheaf monomorphism $\mathcal{A}\longrightarrow W.$

Dually, we want to lift the morphism $W_{1}\longrightarrow\mathcal{B}_{1}$ to a morphism $Q'_{0}\longrightarrow \mathcal{B}_{1}$, where $Q'_{0}=coker(K_{0}(-1,0)\longrightarrow W)$ i.e.

\begin{equation*}
    \xymatrix@C-0.5pc@R-0.5pc{W_{1}\ar@{^{(}->}[d]\ar@{->>}[r]&\mathcal{B}_{1} \\
    Q'_{0}\ar@{-->}[ru]&
    }
\end{equation*}

We also want a surjective morphism $\Hom(Q'_{0},\mathcal{B}_{1})\longrightarrow \Hom(W_{1},\mathcal{B}_{1})$ in order to do the lift. In this case the obstruction is in the group $\ext^{1}(L_{0}(1,0),\mathcal{B}_{1})$ which also vanishes since $\ext^{1}(L_{0}(1,0),\mathcal{B}_{1})\cong\oplus_{i=1}^{n}B_{i}\otimes L^{\ast}_{0}\otimes \h^{1}(\tilde{\mathbb{P}}, \mathcal{O}(0,-E_{i}))=0.$ This means that all the extensions
$$0\longrightarrow\mathcal{B}_{1} \longrightarrow \mathcal{B}\longrightarrow L_{0}(1,0)\longrightarrow0$$
split, hence $\mathcal{B}\cong L_{0}(1,0)\oplus\mathcal{B}_{1}.$ Furthermore we have an epimorphism $W\longrightarrow \mathcal{B}$.
Consequently we have a monad
\begin{equation*}
    M:\quad 0\longrightarrow K_{0}(-1,0)\oplus\mathcal{A}_{1}\longrightarrow W\longrightarrow L_{0}(1,0)\oplus\mathcal{B}_{1}\longrightarrow0
\end{equation*}
with cohomology $\mathcal{F}$ and the following associated display

\begin{equation}\label{M'0}
\xymatrix@C-1.2pc@R-1.3pc{         &     0       \ar[d]           &      0      \ar[d]        &                               &         \\
           & K_{0}(-1,0)\oplus\mathcal{A}_{1} \ar@{=}[r]\ar[d] & K_{0}(-1,0)\oplus\mathcal{A}_{1}  \ar[d] &                 &         \\
          0 \ar[r] &     X       \ar[r]\ar[d]     & W \ar[r] \ar[d] & L_{0}(1,0)\oplus\mathcal{B}_{1} \ar[r] \ar@{=}[d] &    0    \\
          0 \ar[r] & \mathcal{F} \ar[r]\ar[d]     &      Q     \ar[r] \ar[d] & L_{0}(1,0)\oplus\mathcal{B}_{1} \ar[r]      & 0 \\
                   &     0                        &      0                    &    &       \\
}
\end{equation}
from which we compute the Chern character of the cohomology $\mathcal{F};$
\begin{align}
ch(\mathcal{F})&=ch(Q)-ch(L_{0}(1,0))-ch(\mathcal{B}_{1})=ch(X)-ch(K_{0}(-1,0))-ch(\mathcal{A}_{1})\notag \\
&=rk(W)-D-D'-k_{0}-l_{0}-[(D'-D-k_{0}+l_{0})H-\Sigma_{i=1}^{n}(d_{i}+d'_{i})E_{i}]-\frac{(l_{0}+k_{0})}{2}\omega\notag
\end{align}
and by using the relations:
\begin{align}
&rk(W)=r+D+D'+k_{0}+l_{0}, \qquad &k_{0}+l_{0}=2k+|\overrightarrow{a}|^{2}\notag \\
&\Sigma_{i=1}^{n}(d_{i}-d'_{i})=-\Sigma_{i=1}^{n}a_{i}=-\bar{a} \qquad &k_{0}-l_{0}=\bar{a}=-(D-D')
\end{align}
we get
$$ch(\mathcal{F})=r+\Sigma_{i=1}^{n}a_{i}E_{i}-(k+\frac{|\overrightarrow{a}|^{2}}{2})\omega=ch(\mathcal{E})$$

One can also use the three displays of $M_{1}$, $M_{0}$  and $M$ to see that $\mathcal{F}\cong\mathcal{E}.$ Thus, a similar result to \cite[Proposition 1.8]{Buch} is given by the following :
\begin{thm}\label{thm2}
Let $\mathcal{E}$ be a framed torsion-free sheaf with Chern character $ch(\mathcal{E})=r+\Sigma_{i=1}^{n}a_{i}E_{i}-(k+\frac{|\overrightarrow{a}|^{2}}{2})\omega$ on a multi-blow-up, $\pi:\tilde{\mathbb{P}}\longrightarrow\mathbb{P}^{2}$, of $\mathbb{P}^{2}$ at $n$ distinct points. Then $\mathcal{E}$ is given by the cohomology of a monad:
$$M:\quad 0\longrightarrow K_{0}(-1,0)\oplus\mathcal{A}_{1}\longrightarrow W\longrightarrow L_{0}(1,0)\oplus\mathcal{B}_{1}\longrightarrow0$$
where the first map in $M$ is injective as a sheaf map and the second map is surjective, and where
\begin{align}
&A_{i}:=\ext^{1}(\mathcal{E},\mathcal{O}|_{E_{i}}(-1))^{\ast},\qquad B_{i}:=\Hom(\mathcal{E},\mathcal{O}|_{E_{i}}(-1))^{\ast} \notag \\
&\mathcal{A}_{1}:=\oplus_{i=1}^{n} A_{i}(-1,E_{i}), \qquad\quad \mathcal{B}_{1}:=\oplus_{i=1}^{n} B_{i}(1,-E_{i})\notag\\
&K_{0}:=\h^{1}(\tilde{\mathbb{P}}, \mathcal{E}(-2,\overrightarrow{1})),\qquad\quad L_{0}:=\h^{1}(\tilde{\mathbb{P}}, \mathcal{E}(-1,0)). \notag
\end{align}
\end{thm}

\bigskip

\begin{rmk}
\begin{itemize}
\item[(i)] \textnormal{Using the display of the monad $M_{1}$ we can write
\begin{align}
&K_{0}:=\h^{1}(\tilde{\mathbb{P}}, \mathcal{E}(-2,\overrightarrow{1}))\oplus \h^{1}(\tilde{\mathbb{P}}, \mathcal{A}_{1}(-2,\overrightarrow{1}))\oplus \h^{1}(\tilde{\mathbb{P}}, \mathcal{B}_{1}(-2,\overrightarrow{1})) \notag \\
&L_{0}:=\h^{1}(\tilde{\mathbb{P}}, \mathcal{E}(-1,0))\oplus \h^{1}(\tilde{\mathbb{P}}, \mathcal{A}_{1}(-1,0))\oplus \h^{1}(\tilde{\mathbb{P}}, \mathcal{B}_{1}(-1,0)). \notag
\end{align}
From the Riemann-Roch formula one has $\h^{1}(\tilde{\mathbb{P}}, \mathcal{A}_{1}(-2,\overrightarrow{1}))\cong\oplus_{i=1}^{n}A_{i}\otimes \h^{1}(\tilde{\mathbb{P}}, \mathcal{O}(-3,\overrightarrow{1}+E_{i}))=0$ and also $\h^{1}(\tilde{\mathbb{P}}, \mathcal{B}_{1}(-2,\overrightarrow{1}))\cong\oplus_{i=1}^{n}B_{i}\otimes \h^{1}(\tilde{\mathbb{P}}, \mathcal{O}(-1,\overrightarrow{1}-E_{i}))=0.$
The vanishing holds also for $\h^{1}(\tilde{\mathbb{P}}, \mathcal{A}_{1}(-1,0))$ and $\h^{1}(\tilde{\mathbb{P}}, \mathcal{B}_{1}(-1,0)).$ Hence we have the forms of $K_{0}$ and $L_{0}$ given in the theorem above.}

\item[(ii)]\textnormal{The fact that we restrict to blow-ups of $\mathbb{P}^{2}$ made at distinct points, excluding the case of iterated blow-ups, is a technical assumption, and not a conceptual one. Without it the construction of explicit ADHM type data would be considerably more complicated. In this general case, one can have a monad, associated to $\mathcal{E}$ with the same conditions as in the theorem above, in which the first and last terms are non-trivial extensions of the form $$\xymatrix@C-0.5pc@R-2pc{0\ar[r]&K_{0}(-1,0)\ar[r]&\mathcal{A}\ar[r]&\mathcal{A}_{1}\ar[r]&0 \\ 0\ar[r]&\mathcal{A}_{1}\ar[r]&\mathcal{B}\ar[r]&L_{0}(1,0)\ar[r]&0}$$ rather then being direct sums, as given by \cite[Proposition 1.5]{Buch} in the case of bundles. This is because, for any $i,$ one has $\h^{1}(\tilde{\mathbb{P}},\mathcal{O}(0,-E_{i}))\neq0.$ So in what follows we will still be considering only the case of multiple blow-ups at distinct points.
}
\end{itemize}
\end{rmk}

\bigskip

In order to study the moduli space of framed torsion-free sheaves on $\tilde{\mathbb{P}}$ through monads, we need to know if families of monads of the type we are considering behave well in describing families of torsion-free sheaves; we start by reminding the following:

\begin{pr}
Let $M:0\longrightarrow\mathcal{A}\longrightarrow \mathcal{W}\longrightarrow \mathcal{B}\longrightarrow0$ and $M':0\longrightarrow\mathcal{A}'\longrightarrow \mathcal{W}'\longrightarrow \mathcal{B}'\longrightarrow0$ be two monads on a surface $\mathcal{S}$ with cohomologies $\mathcal{E}$ and $\mathcal{E}'$ respectively. The morphism $H:\Hom(M,M')\longrightarrow \Hom(\mathcal{E},\mathcal{E}')$ is surjective if
$$\ext^{1}(\mathcal{B},\mathcal{W}')=\ext^{1}(\mathcal{W},\mathcal{A}')=\ext^{2}(\mathcal{B},\mathcal{A}')=0.$$
Furthermore its kernel is identified with $\ext^{1}(\mathcal{B},\mathcal{A}')$ if $\Hom(\mathcal{B},\mathcal{W}')=\Hom(\mathcal{W},\mathcal{A}')=0.$
\end{pr}
\begin{proof}
The proof for the locally-free case can be found in \cite[Chapter \textbf{II}, Lemma \textbf{4.1.3}]{Okonek}. In the torsion-free case, the steps of the proof can be repeated provided that one replaces suitably the global section functor $\Gamma$ by the $\Hom$ functor.
%provided that one takes the cohomology groups by suitable $\ext-$groups.
\end{proof}
In our case, one has
\begin{align}
&\ext^{1}(\mathcal{B},W')=0,\quad\quad \ext^{1}(W,\mathcal{A}')=0,\quad\quad \ext^{2}(\mathcal{B},\mathcal{A}')=0, \notag\\
&\Hom(\mathcal{B},W')=0, \quad\quad \Hom(W,\mathcal{A}')=0,\notag
\end{align}
by using the Riemann-Roch theorem. Hence $H:\Hom(M,M')\longrightarrow \Hom(\mathcal{E},\mathcal{E}')$ is surjective with kernel $\ext^{1}(\mathcal{B},\mathcal{A}')\cong\oplus_{i=1}^{n} B^{\ast}_{i}\otimes A_{i}'$. This means that the functor
\begin{equation}\label{functor}
\mathfrak{H}:\mathfrak{M}on\longrightarrow\mathfrak{F}ram,
\end{equation}
from the category of monads given by theorem ${\bf\ref{thm2}}$ to the category of framed torsion-free sheaves on $\tilde{\mathbb{P}},$ is full. In the next section we will get rid of the kernel of $H:=\mathfrak{H}(M,M')$ by reducing the monad, so that the corresponding reduced functor is fully faithfull. We also remark that since the operation of taking cohomology is functorial, then the morphisms of the monad of theorem ${\bf\ref{thm2}}$ are natural in the family of framed torsion-free sheaves.

The monad we constructed describes well families of framed torsion-free sheaves and one can talk about a moduli space of such objects, but we still have to fix the problem of the control on the dimensions of $A_{i}$ and $B_{i}$; their difference is constant $dimA_{i}-dimB_{i}=a_{i}$, but each dimension can, a priori, jump. Proceeding as in \cite[Secion {\bf 1}]{Buch}, we apply the functor $\Hom(\mathcal{E},\cdot)$ to the sequence \eqref{rest2} twisted by $\mathcal{O}(-1,E_{i})$ :
$$0\longrightarrow B_{i}^{\ast}\longrightarrow \ext^{1}(\mathcal{E},\mathcal{O}(-1,0))\longrightarrow \ext^{1}(\mathcal{E},\mathcal{O}(-1,E_{i}))\longrightarrow A_{i}^{\ast}\longrightarrow0.$$
We split this sequence by defining $V_{i}^{\ast}=ker(\ext^{1}(\mathcal{E},\mathcal{O}(-1,E_{i}))\longrightarrow A_{i}^{\ast}),$ thus getting
$$0\longrightarrow B_{i}^{\ast}\longrightarrow \ext^{1}(\mathcal{E},\mathcal{O}(-1,0))\longrightarrow V_{i}^{\ast}\longrightarrow0$$
$$0\longrightarrow V_{i}^{\ast}\longrightarrow \ext^{1}(\mathcal{E},\mathcal{O}(-1,E_{i}))\longrightarrow A_{i}^{\ast}\longrightarrow0$$
dualizing the sequences and using the fact that
$$\ext^{1}(\mathcal{E},\mathcal{O}(-1,0))^{\ast}=\ext^{1}(\mathcal{O}(-1,0),\mathcal{E}(-3,\overrightarrow{1}))
=\h^{1}(\tilde{\mathbb{P}}, \mathcal{E}(-2,\overrightarrow{1}))\quad\textrm{and}$$
$$\ext^{1}(\mathcal{E},\mathcal{O}(-1,E_{i}))^{\ast}=\ext^{1}(\mathcal{O}(-1,E_{i}),\mathcal{E}(-3,\overrightarrow{1}))
=\h^{1}(\tilde{\mathbb{P}}, \mathcal{E}(-2,\overrightarrow{1}-E_{i}))$$
which have dimensions, respectively, $k+\frac{1}{2}\Sigma_{j=1}^{n}a_{j}(a_{j}+1)$ and $k+\frac{1}{2}\Sigma_{j=1}^{n}a_{j}(a_{j}+1)-a_{i}$,
one has the isomorphisms :
\begin{equation}\label{isom}
A_{i}\oplus V_{i}\cong \h^{1}(\tilde{\mathbb{P}}, \mathcal{E}(-2,\overrightarrow{1}-E_{i})),\qquad B_{i}\oplus V_{i}\cong \h^{1}(\tilde{\mathbb{P}}, \mathcal{E}(-2,\overrightarrow{1})).
\end{equation}
On the other hand consider, for each $i,$ the following extension
\begin{equation}\label{W0}
0\longrightarrow \mathcal{O}(-1,E_{i})\longrightarrow W_{0}\longrightarrow\mathcal{O}(1,-E_{i})\longrightarrow0
\end{equation}
The group classifying such extensions is $\ext^{1}(\mathcal{O}(1,-E_{i}),\mathcal{O}(-1,E_{i}))\cong\h^{1}(\tilde{\mathbb{P}},\mathcal{O}(-2, 2E_{i}))\cong\C$ whose dimension is 1. Taking the non-trivial extension in this group we show that $W_{0}$ is a trivial 2-bundle: the extension obtained after restricting to the generic line $H$ corresponds to a non zero element in $\h^{1}(H,\mathcal{O}_{H}(-2))=\C$, since there is an isomorphism $\h^{1}(\tilde{\mathbb{P}},\mathcal{O}(-2, 2E_{i}))\cong\h^{1}(H,\mathcal{O}_{H}(-2)),$ furthermore it is a twisted Euler sequence $0\longrightarrow\mathcal{O}_{H}(-1)\longrightarrow\mathcal{O}_{H}^{\oplus2}\longrightarrow T_{H}(-1)\longrightarrow0$ associated to the line $H.$ The twisting is given by $\mathcal{O}_{H}(-1)$ and the twisted tangent space $T_{H}(-1)$ is identified with $\mathcal{O}_{H}(1).$ Thus $W_{0}|_{H}$ is the trivial bundle $\mathcal{O}_{H}^{\oplus2}$. Restricting the extension \eqref{W0} to the exceptional divisor $E_{j}$, one concludes, in the same way as above, that if $i=j$ then $W_{0}|_{E_{j}}$ corresponds to a twisted Euler sequence associated to the line $E_{i},$ where the twisting is given by $\mathcal{O}_{E_{i}}(-1),$ hence $W_{0}|_{E_{j}}$ is the trivial 2-bundle $\mathcal{O}_{E_{j}}^{\oplus2}$. When $i\neq j$ it is obvious that the restriction is a trivial 2-bundle. Consequently $W_{0}$ is a trivial 2-bundle on $\tilde{\mathbb{P}}$ since its Chern classes are both zero.

Now we have to twist the extension \eqref{W0} by $V_{i}$;
\begin{equation}
0\longrightarrow V_{i}(-1,E_{i})\longrightarrow V_{i}\otimes W_{0}\longrightarrow V_{i}(1,-E_{i})\longrightarrow0
\end{equation}
but since adding such exact sequences, for every $i$, to the monad in the theorem \textbf{\ref{thm2}} will not change the cohomology $\mathcal{E},$ and moreover, by using the isomorphisms \eqref{isom} one has a generalization of \cite[Prposition 1.10]{Buch} given by the following
\begin{pr}\label{prop}
Let $\mathcal{E}$ be a torsion-free sheaf satisfying the conditions of theorem \textbf{\ref{thm2}}. Then there exists a monad $M$ of the form:
$$\xymatrix@R-1pc{M:& 0\ar[r]& \oplus_{i=0}^{n}K_{i}(-1,E_{i})\ar[r]^{\qquad\quad\alpha}& W \ar[r]^{\beta\qquad\quad}&\oplus_{i=0}^{n}L_{i}(1,-E_{i})\ar[r]&0}$$
whose cohomology is $\mathcal{E},$ and in which $\alpha$ is injective as a sheaf morphism and $\beta$ is surjective. Moreover
$$\left\{\begin{array}{l} K_{i}=\h^{1}(\tilde{\mathbb{P}}, \mathcal{E}(-2,\overrightarrow{1}-E_{i})) \\ L_{i}=\h^{1}(\tilde{\mathbb{P}}, \mathcal{E}(-2,\overrightarrow{1})) \end{array}\right. i\neq0$$
$$\left\{\begin{array}{l} K_{0}:=\h^{1}(\tilde{\mathbb{P}}, \mathcal{E}(-2,\overrightarrow{1})) \\ L_{0}:=\h^{1}(\tilde{\mathbb{P}}, \mathcal{E}(-1,0)). \end{array}\right.\qquad\quad$$ and $E_{0}:=0.$
\end{pr}

\bigskip\bigskip

\begin{rmk}\label{remark}
\begin{itemize}
\item[(i)]  \textnormal{In the proposition above we have used the same notation as in theorem \textbf{\ref{thm2}}, while the spaces and the monads are different.}

\item[(ii)] \textnormal{Since the dimensions of our spaces are as follows:
\begin{align}
&dim K_{0}=k+\frac{|a|^{2}+\bar{a}}{2},\qquad dim L_{0}=k+\frac{|a|^{2}-\bar{a}}{2}\notag\\
&dim K_{i}=k+\frac{|a|^{2}+\bar{a}}{2}-a_{i},\qquad dim L_{i}=k+\frac{|a|^{2}+\bar{a}}{2}\notag
\end{align}
then $$dim K_{0}+\Sigma_{i=1}^{n}dim K_{i}=(n+1)k+\frac{n+1}{2}(|a|^{2}+\bar{a})-\bar{a}$$ and $$dim L_{0}+\Sigma_{i=1}^{n}dim L_{i}=(n+1)k+\frac{n+1}{2}|a|^{2}+\frac{n-1}{2}\bar{a}=(n+1)k+\frac{n+1}{2}(|a|^{2}+\bar{a})-\bar{a},$$ thus the left and right hand terms in the monad have the same rank $dim L_{0}+\Sigma_{i=1}^{n}dim L_{i}.$
Hence  $rkW=2(dim L_{0}+\Sigma_{i=1}^{n}dim L_{i})+r.$}

\item[(iii)] \textnormal{Using the display of the monad above one can see that the kernel $X$ of the map $\beta$ is a locally free sheaf (since it is the kernel of a bundle map). The fact that the cohomology is a torsion-free sheaf implies that the first map $\alpha$ fails to be of maximal rank on finitely many points. These are the singularity set of the torsion-free sheaf that we are describing. Conversely if the map $\alpha$ is not of maximal rank on some finite set of points, then the cohomology of the monad would be a torsion-free sheaf with singularity set, exactly, the set of the points where $\alpha$ vanishes.}
\end{itemize}
\end{rmk}

\section{The ADHM data}\label{ADHM}

In this section we describe briefly the ADHM data associated to the monad of proposition \textbf{\ref{prop}}. This will lead to a presentation of the moduli space of framed torsion-free sheaves under study as quotient of a space of some matrices satisfying some constraints by the action of an algebraic group. We now summarize some notations due to Penrose and Buchdahl (see \cite[Section \textbf{2}]{Buch}).

We denote the homogeneous coordinates on $\mathbb{P}^{2}$ by $(z^{0},z^{1},z^{2})$. The line at infinity $l_{\infty}$ is given by the equation $z^{2}=0$. Let $p_{i}$ be one of the points in $\mathbb{P}^{2}$ at which we perform a blow-up, and assume $p_{i}\notin l_{\infty}$  for all $i$. So all these points are in the chart $[z^{0},z^{1},1]$ and denote their inhomogeneous coordinates there by $p^{a}_{i}=(p^{A}_{i},1)$ where $A=0,1$ and $a=0,1,2$. Locally, near every blow-up point, $\tilde{\mathbb{P}}$ can be described by $\{([z^{0},z^{1},z^{2}],[w_{i}^{0},w_{i}^{1}])\in\mathbb{P}^{2}\times\mathbb{P}^{1} \mid (z^{0}-p^{0}_{i}z^{2})w^{1}_{i}=(z^{1}-p^{1}_{i}z^{2})w^{0}_{i}\}.$
Each blow-up equation
\begin{equation}\label{blowup}
(z^{0}-p^{0}_{i}z^{2})w^{1}_{i}=(z^{1}-p^{1}_{i}z^{2})w^{0}_{i}, \qquad\qquad (w^{0},w^{1})\in\mathbb{C}^{2}\setminus\{0\},
\end{equation}
is satisfied if and only if there exists a unique scalar $\lambda_{i}$ such that
\begin{equation}\label{blowup1}
\lambda_{i}w_{i}^{1}=z^{1}-p^{1}_{i}z^{2}\quad\textnormal{and}\quad\lambda_{i}w_{i}^{0}=z^{0}-p^{0}_{i}z^{2}.
\end{equation}
Moreover $z^{0}-p^{0}_{i}z^{2}$ and $z^{1}-p^{1}_{i}z^{2}$ are both sections of $\mathcal{O}(H)$ vanishing on $E_{i},$ while $w_{i}^{0}$ and $w_{i}^{1}$ can be viewed as sections of $\mathcal{O}(H-E_{i}).$ Thus the homogeneous equation \eqref{blowup} uniquely determines a section $\lambda_{i}$ of a line bundle $\mathcal{O}(E_{i})$ over the blow-up such that \eqref{blowup1} is satisfied. Note also that the restriction of $w^{A}_{i}$ to the exceptional divisor $E_{i}$ gives the homogeneous coordinates $(w^{0}_{i},w^{1}_{i})$ of $E_{i}$.
Finally we denote $$\left\{\begin{array}{ll}w_{i0}:=-w^{1}_{i} & \\ w_{i1}:=w^{0}_{i}&\end{array}\right.$$ Assuming the convention of summing over repeated upper and lower indices, it follows that $w^{A}_{i}w_{iA}=0.$ More generally if we have a pair of matrices as $m_{A}=(m_{0},m_{1})$ we use the same two index notations as above, and we write $m^{A}m_{A}=m_{0}m_{1}-m_{1}m_{0}.$ For the direct sum $V\oplus V$ and the morphism $(m_{0},m_{1}):V\longrightarrow V\oplus V$ we use the notation $m_{A}:V\longrightarrow V_{A}$. Given a morphism $V\oplus V\longrightarrow V$ where $(v_{0},v_{1})\longrightarrow m^{0}v_{0}+m^{1}v_{1}$, this would be written as $m^{A}:V_{A}\longrightarrow V$.

\bigskip\bigskip
Let us consider a monad $M$ as in Proposition \textbf{\ref{prop}}
$$\xymatrix{M:&0\ar[r]&\oplus_{i=0}^{n}K_{i}(-1,E_{i})\ar[r]^{\quad\quad\alpha}& W\ar[r]^{\beta\quad\quad}& \oplus_{i=0}^{n}L_{i}(1,-E_{i})\ar[r]&0}$$
with cohomology a framed torsion-free sheaf $\mathcal{E}$. The map $\beta$ is surjective everywhere on $\tilde{\mathbb{P}},$ while the map  $\alpha$ is injective except at some finite set of points in $\tilde{\mathbb{P}}.$ The condition $\beta\circ\alpha=0$ will give a set of matrix equations. At this point we can reduce the linear data, associated to the monad above and the automorphism group acting on it, in a similar way to reference \cite[section \textbf{3}]{Buch}, to which we refer for more details, since the difference between the above monad and the one given in \cite{Buch} resides only in the condition imposed on the rank of the map $\alpha$; in our case the rank of $\alpha$ is maximal except at finitely many points. It follows that the moduli space $\mathcal{M}^{\tilde{\mathbb{P}}}_{\vec{a}, k}$ of framed torsion-free sheaves $\mathcal{E}$, on $\tilde{\mathbb{P}}$, with Chern character $ch(\mathcal{E})=r+\Sigma_{i=1}^{n}a_{i}E_{i}-(k+\frac{|\vec{a}|^{2}}{2})\omega$, is identified with the quotient $P'/G' $, where $P'$ is the space of configurations $(a,a_{00}^{A},c,d)$ defined as follows: $a$ is a non-singular matrix of the form
\begin{equation}
a=\left[\begin{array}{lllll}a_{00}&a_{01}&a_{02}&\cdots&a_{0n}\\
a_{10}&a_{11}&0&\cdots&0\\
a_{20}&0&a_{22}&\cdots&0\\
\vdots&\vdots&\vdots&\ddots&\vdots\\
a_{n0}&0&0&\cdots&a_{nn}\\
\end{array}\right]
\end{equation}
where $a_{ij}\in \Hom(K_{j},L_{i}),$  $c\in \Hom(K_{0},\mathbb{C}^{r})$, $d\in \Hom(\mathbb{C}^{r},L_{0})$ and $a_{00A}\in \Hom(K_{0},L_{0}),$ for $A=0,1$. The configuration $(a,a_{00}^{A},c,d)$ satisfies the equation
\begin{equation}\label{constr}
(q^{A}a^{-1}q_{A})^{00}+dc=0
\end{equation}
where the $00$ subscript means the $00$ entry and the matrix $q^{A}$ is defined as
\begin{equation}
q^{A}=\left[\begin{array}{lllll}-a_{00}^{A}&p^{A}_{1}a_{01}&p^{A}_{2}a_{02}&\cdots&p^{A}_{n}a_{0n} \\
p^{A}_{1}a_{10}&p^{A}_{1}a_{11}&0 &\cdots&0 \\
p^{A}_{2}a_{20}&0&p^{A}_{2}a_{22}&\cdots&0 \\
\vdots&\vdots&\vdots&\ddots&\vdots \\
p^{A}_{n}a_{n0}&0&0&\cdots&p^{A}_{n}a_{nn}
\end{array}\right]
\end{equation}
The group $G'$ consists of the non-singular transformations of the form
\begin{equation}
g=\left[\begin{array}{lllll}g_{00}&g_{01}&g_{02}&\cdots&g_{0n}\\
0&g_{11}&0&\cdots&0\\
0&0&g_{22}&\cdots&0\\
\vdots&\vdots&\vdots&\ddots&\vdots\\
0&0&0&\cdots&g_{nn}\\
\end{array}\right], \quad h=\left[\begin{array}{lllll}h_{00}&0&0&\cdots&0\\
h_{10}&h_{11}&0&\cdots&0\\
h_{20}&0&h_{22}&\cdots&0\\
\vdots&\vdots&\vdots&\ddots&\vdots\\
h_{n0}&0&0&\cdots&h_{nn}\\
\end{array}\right]
\end{equation}

The set of transformations that fix the configuration $(a,q^{A},c,d)$ are of the form $g=1+am$, $h=(1+ma)^{-1}$, where the matrix $m$ is of the form
\begin{equation}
m=\left[\begin{array}{lllll}0&0&0&\cdots&0\\
0&m_{11}&0&\cdots&0\\
0&0&m_{22}&\cdots&0\\
\vdots&\vdots&\vdots&\ddots&\vdots\\
0&0&0&\cdots&m_{nn}\\
\end{array}\right]\in \Hom(\oplus_{i=0}^{n}L_{i}, \oplus_{i=0}^{n}K_{i})
\end{equation}

Fixing an isomorphism $L_{i}\cong K_{0}$, one can choose the matrix $a$ so that $a_{0i}=1 \quad \forall i>0$. To preserve this form of $a$, the transformation $g$ must be of the above form and such that $g_{ii}=(h_{00}+a_{ii}h_{i0})^{-1}$. One checks that for every matrix $h$ such that $h_{00}+a_{ii}h_{i0}$ is non-singular, there exists a matrix $D$ of the form $D=diag(d_{00},d_{11},\cdots,d_{nn})$ such that $h=(1+ma)^{-1}D,$ where $a$ of the above form. We remark that the space of matrices $m,$ as above, is exactly the kernel of $H:\End(M)\longrightarrow \End(\mathcal{E}).$ Moreover if we consider only actions of $(g,h)$ having the form
\begin{equation}\label{group-form}
h=diag(h_{00},h_{11},\cdots,h_{nn}),\quad\quad
g=\left[\begin{array}{lllll}g_{00}&g_{01}&g_{02}&\cdots&g_{0n}\\
0&h^{-1}_{00}&0&\cdots&0\\
0&0&h^{-1}_{00}&\cdots&0\\
\vdots&\vdots&\vdots&\ddots&\vdots\\
0&0&0&\cdots&h^{-1}_{00}\\
\end{array}\right].
\end{equation}
then one has a free action since the isotropy subgroup, for a given configuration, is essentially the discarded transformations $m$. Finally the moduli space $\mathcal{M}^{\tilde{\mathbb{P}}}_{\vec{a}, k}$ has another presentation as the quotient $P/G $, where
\begin{equation}\label{spaceP}
P=\left\{\begin{array}{l}\textnormal{configurations } (a,a_{00}^{A},c,d) \\ \textnormal{ as above} \end{array}| \begin{array}{l}a_{i0} \textnormal{ is equal to the identity matrix } \mathbb{I}_{K_{0}}  \textnormal{ on } K_{0} \\  \textnormal{ for each } i>0 \textnormal{ and satisfying the equation \eqref{constr}}\end{array}\right\}
\end{equation}
The group $G$ consists of the non-singular transformations of the form \eqref{group-form} acting explicitly as follows
\begin{align}\label{P and G}
&a_{00}\longrightarrow g_{00}(a_{00}+\Sigma_{i=1}^{n}g^{-1}_{00}g_{0i})h_{00}, \quad a^{A}_{00}\longrightarrow g_{00}(a^{A}_{00}-\Sigma_{i=1}^{n}g^{-1}_{00}g_{0i}p_{i}^{A})h_{00}  \\
&a_{ii}\longrightarrow h^{-1}_{00}a_{ii}h_{ii}, \quad a_{0i}\longrightarrow g_{00}(a_{0i}+g^{-1}_{00}g_{0i}a_{ii})h_{ii}, \quad i>0, \quad c\longrightarrow ch_{00}, \quad d\longrightarrow g_{00}d.\notag
\end{align}

The above free action of the group, obtained from the reduction of the monad, has isotropy subgroup $G_{\mathcal{C}}=\{Id\},$ for any configuration $\mathcal{C}=(a,a_{00}^{A},c,d)$. This will give a nonsingular quotient as we shall prove later, and one can easily check that the dimension of this moduli space is $dim\mathcal{M}^{\tilde{\mathbb{P}}}_{\vec{a}, k}=2r(k+\frac{|\vec{a}|^{2}}{2})-|\vec{a}|^{2}$.

One can reconstruct the monad from the data above by the following;
\begin{equation}\label{alpha}
\alpha=\left[\begin{array}{lllll}a_{00}z_{A}+a_{00A}z^{2}&a_{01}w_{1A}&a_{02}w_{2A}&\cdots&a_{0n}w_{nA}\\
\mathbb{I}_{K_{0}}\lambda_{1} w_{1A}&a_{11}w_{1A}&\quad0&\cdots&\quad0\\
\mathbb{I}_{K_{0}}\lambda_{2} w_{2A}&\quad0&a_{22}w_{2A}&\cdots&\quad0\\
\quad\vdots&\quad\vdots&\quad\vdots&\ddots&\quad\vdots\\
\mathbb{I}_{K_{0}}\lambda_{n} w_{nA}&\quad0&\quad0&\cdots&a_{nn}w_{nA}\\
cz^{2}&\quad0&\quad0&\cdots&\quad0
\end{array}\right]
\end{equation}

\begin{equation}\label{beta}
\beta=\left[\begin{array}{llllll}z^{A}+b^{A}_{00}z^{2}&b^{A}_{01}z^{1}&b^{A}_{02}z^{2}&\cdots& b^{A}_{0n}z^{n}& dz^{2}\\
\quad0&w^{1A}&0&\cdots&\quad0&0\\
\quad0&0&w^{2A}&\cdots&\quad0&0\\
\quad\vdots&\vdots&\vdots&\ddots&\quad\vdots&\vdots\\
\quad0&0&0&\cdots&w^{nA}&0\end{array}\right]
\end{equation}

Here the $b^{A}_{0i}$ are given by the following: we define the matrix $b^{A}=\tilde{a}^{A}a^{-1},$ where $\tilde{a}^{A}:=\left[\begin{array}{lllll}a^{A}_{00}&-p^{A}_{1}&-p^{A}_{2}&\cdots&-p^{A}_{n}\end{array}\right].$ We would like to remind that $p^{A}_{i}$'s are the inhomogeneous coordinates, in $\mathbb{P}^{2},$ of the blow-up points. Then $b^{A}$ is of the form $b^{A}:=\left[\begin{array}{llll}b^{A}_{00}&b^{A}_{01}&\cdots&b^{A}_{0n}\end{array}\right].$ Recall that the non-degeneracy conditions imposed on the data are the surjectivity of $\beta$ everywhere on $\tilde{\mathbb{P}}$, and injectivity of $\alpha$ generically, i.e., $\alpha$ can be non-injective only at finitely many points in $\tilde{\mathbb{P}}.$

Finally we remark that Since the isotropy subgroups were discarded when we reduced the monad, and since that was exactly the kernel of the map $H:\End(M,M)\longrightarrow\End(\mathcal{E}),$ then the functor $\mathfrak{H}$, in \eqref{functor}, from the category of torsion-free sheaves on $\tilde{\mathbb{P}}$ to the category of monads obtained from \eqref{alpha} and \eqref{beta} is fully faithfull.

\section{Moduli functor and universal monads}\label{unversal}

In the following we want to introduce some necessary material which will enable us to prove, in the last section of this paper, that the moduli space $\mathcal{M}^{\tilde{\mathbb{P}}}_{\vec{a}, k}$ is fine. As a first step we consider families of framed torsion-free sheaves, parameterized by some scheme $S.$ For each family we shall build a monad $\mathcal{M}$ on the product $\tilde{\mathbb{P}}\times S$ whose cohomology is the family we started with. But before starting this program let us give some definitions and set some notations for our moduli functor problem.

Let $\mathfrak{M}^{\tilde{\mathbb{P}}}_{\vec{a}, k}: \mathfrak{S}ch \longrightarrow \mathfrak{S}et$ be the contravariant functor from the category of noetherian schemes of finite type to the category of sets, which is defined as follows;
to every such scheme $S$ we associate the set $$\mathfrak{M}^{\tilde{\mathbb{P}}}_{\vec{a}, k}(S)=\left\{[\mathcal{F}] \Large{\textbf{/}}\small{ \begin{array}{l} \quad\mathcal{F} \textrm{ is a coherent sheaf on }\tilde{\mathbb{P}}\times S,\hspace{0.1cm}\textrm{flat on } S, \textrm{ and such that}\\ \mathcal{F}\otimes k(s)\cong\mathcal{E}\textrm{ is a framed torsion-free sheaf on } \tilde{\mathbb{P}}\textrm{ with }\\ \quad\quad\quad ch(\mathcal{E})=r+\Sigma_{i=1}^{n}a_{i}E_{i}-(k-\frac{|\vec{a}|^{2}}{2})\omega \end{array}} \right\}$$
where $[\mathcal{F}]$ stands for the class of the sheaf $\mathcal{F}$ under the following equivalence: $\mathcal{F}'\in[\mathcal{F}]$ if there is a line bundle $L$ on $S$ such that $\mathcal{F}'\cong\mathcal{F}\otimes\tilde{q}^{\ast}L.$ The morphism $\tilde{q}$ is the projection $\tilde{\mathbb{P}}\times S\longrightarrow S.$ We also denote by $k(s)$ the residue field of a point $s\in S.$
\begin{definition}\cite[Chapter \textbf{I}, $\S$ \textbf{2}]{newstead}
A scheme ${\mathcal{M}^{\tilde{\mathbb{P}}}_{\vec{a}, k}}$ is called a coarse moduli space if the following conditions are satisfied:
\begin{itemize}
\item There is a natural transformation $$\Phi: \mathfrak{M}^{\tilde{\mathbb{P}}}_{\vec{a}, k}(\bullet)\longrightarrow \Hom(\bullet,\mathcal{M}^{\tilde{\mathbb{P}}}_{\vec{a}, k})$$ which is a bijection for every closed point $s\in S$.

\item For every scheme $\mathcal{R}$ and every natural transformation $$\Psi: \mathfrak{M}^{\tilde{\mathbb{P}}}_{\vec{a}, k}(\bullet)\longrightarrow \Hom(\bullet,\mathcal{R})$$ there is a unique morphism of schemes $f:\mathcal{M}^{\tilde{\mathbb{P}}}_{\vec{a}, k}\longrightarrow \mathcal{R}$ such that the diagram of natural transformations
\begin{equation}
\xymatrix@C-0.8pc@R-0.9pc{\mathfrak{M}^{\tilde{\mathbb{P}}}_{\vec{a}, k}(\bullet)\ar[r]\ar[rd]& \Hom(\bullet,\mathcal{M}^{\tilde{\mathbb{P}}}_{\vec{a}, k})\ar[d]^{f_{\ast}} \\
&\Hom(\bullet,\mathcal{R})
}
\end{equation}
commutes.
\end{itemize}
\end{definition}

\begin{definition}\cite[Chapter \textbf{I}, $\S$ \textbf{2}]{newstead}
A scheme $\mathcal{M}^{\tilde{\mathbb{P}}}_{\vec{a}, k}$ is called a fine moduli space for the functor $\mathfrak{M}^{\tilde{\mathbb{P}}}_{\vec{a}, k}(\bullet)$ if the above natural transformation $\Phi$ is an isomorphism.
\end{definition}

Given a noetherian scheme of finite type $S$, let $\mathcal{F}$ be a coherent sheaf on $\tilde{\mathbb{P}}\times S$ which is flat on $S.$ For every point $s\in S$, $\mathcal{F}_{s}:=\mathcal{F}\otimes k(s)$ is a framed torsion-free sheaf $\mathcal{E}$ on $\tilde{\mathbb{P}}$ with Chern character $ch(\mathcal{E})=r+\Sigma_{i=1}^{n}a_{i}E_{i}-(k-\frac{|\vec{a}|^{2}}{2})\omega .$  We consider the following diagram:
\begin{equation}
\xymatrix@C-0.5pc@R-0.5pc{& \mathcal{F} \ar[d]& \\
&\tilde{\mathbb{P}}\times S\ar[ld]_{\tilde{p}}\ar[rd]^{\tilde{q}}& \\
\tilde{\mathbb{P}}&& S
}
\end{equation}
where $\tilde{p}$ and $\tilde{q}$ are the natural projections onto the two factors. The aim is to construct a monad $\mathcal{M}$ on $\tilde{\mathbb{P}}\times S$ which is associated to the sheaf $\mathcal{F}$, i.e., $Coh(\mathcal{M})=\mathcal{F}$.
For every two sheaves, $\mathcal{F}$ on $\tilde{\mathbb{P}}$ and $\mathcal{G}$ on $S$, we use the following notation: $\mathcal{F}\boxtimes \mathcal{G}:=\tilde{p}^{\ast}\mathcal{F}\otimes\tilde{q}^{\ast}\mathcal{G}.$
We consider also, for a morphism $f:X\longrightarrow Y$ of noetherian schemes of finite type, the functor $\mathcal{H}om_{f}(\mathcal{G}, \bullet):= f_{\ast}\circ\mathcal{H}om(\mathcal{G},\bullet)$ called the relative $\mathcal{H}om$-functor, and we denote its $i$-th right derived functors by $\mathcal{E}xt^{i}_{f}(\mathcal{G}, \bullet)$. For more details see for example \cite{Kleiman}.

We remind the reader that a sheaf $\mathcal{F}$ on a topological space $X$ is \emph{flasque} (or \emph{flabby}) if for every inclusion $V\subseteq U$ of open sets, the restriction map $\mathcal{F}(U)\longrightarrow\mathcal{F}(V)$ is surjective.

\bigskip

Before proceeding we want to introduce the following useful result, which is a kind of "relative local-to-global" spectral sequence:

\begin{pr}\label{loctoglob}
Let $f:X\longrightarrow Y$ be a morphism of noetherian schemes of finite type, and $\mathcal{G}$ a coherent sheaf on $X$ flat on $Y.$ Then for any coherent sheaf $\mathcal{J}$ on $X$ flat on $Y$, there exists a spectral sequence $E_{r}^{p,q}$ with $E_{2}$-term $E_{2}^{p,q}=\h^{p}(Y,\mathcal{E}xt_{f}^{q}(\mathcal{G},\mathcal{J}))$ which converges to $E_{\infty}^{p+q}=\ext^{p+q}_{X}(\mathcal{G},\mathcal{J}).$
\end{pr}
\begin{proof}
Let $\mathcal{J}\longrightarrow I^{\bullet}$ be an injective resolution of the sheaf $\mathcal{J}$, and consider a Cartan-Eilenberg resolution given by the \v Cech complex associated to the complex of sheaves $\mathcal{H}om_{f}(\mathcal{G}, I^{\bullet})$ for a suitable open cover $\mathcal{Y}$ of $Y$. This defines a double complex $C^{\bullet}(\mathcal{Y}, \mathcal{H}om_{f}(\mathcal{G}, I^{\bullet}))$ with differentials $$\left\{\begin{array}{l}\delta_{1}=d \textrm{ the differential associated to the injective resolution of }\mathcal{J} \\ \delta_{2}=\delta \textrm{ the differential associated to the }\check{C}\textrm{ech complex.} \end{array}\right.$$
There are two spectral sequences associated with this double complex \cite{Grothendieck}, with $E_{2}$-terms:
$$'E_{2}^{p,q}=\h^{p}_{d}[\h^{q}_{\delta}(C^{\bullet}(\mathcal{Y}, \mathcal{H}om_{f}(\mathcal{G}, I^{\bullet})))]$$
$$''E_{2}^{p,q}=\h^{p}_{\delta}[\h^{q}_{d}(C^{\bullet}(\mathcal{Y}, \mathcal{H}om_{f}(\mathcal{G}, I^{\bullet})))]$$
The $E_{1}$-term of the first spectral sequence is given by:
\begin{align}
'E_{1}^{p,q}&=\h^{q}_{\delta}(C^{\bullet}(\mathcal{Y}, \mathcal{H}om_{f}(\mathcal{G}, I^{p}))) \notag\\
&=\h^{q}_{\delta}(C^{\bullet}(\mathcal{Y}, f_{\ast}\circ\mathcal{H}om(\mathcal{G}, I^{p}))) \notag
\end{align}
By \cite[Lemma \textbf{II. 7. 3}]{Godement} the sheaf $\mathcal{H}om(\mathcal{G}, I^{p})$ is flasque for every term $I^{p}$ of the injective resolution. Moreover the direct image of a flasque sheaf is flasque \cite[\textbf{II} Theorem \textbf{3.1.1}]{Godement}. Hence $$'E_{1}^{p,q}=\left\{\begin{array}{ll} \h^{0}(Y,f_{\ast}\circ\mathcal{H}om(\mathcal{G}, I^{p}))& q=0 \\ \quad 0 & q\neq0\end{array}\right.$$ On the other hand we have
\begin{align}
\h^{0}(Y,f_{\ast}\circ\mathcal{H}om(\mathcal{G}, I^{p}))&=\h^{0}(X,\mathcal{H}om(\mathcal{G}, I^{p})) \notag \\
&=\Hom(\mathcal{G}, I^{p}) \notag
\end{align}
Thus $$'E_{1}^{p,q}=\left\{\begin{array}{ll} \Hom(\mathcal{G}, I^{p})& q=0\\ \quad 0 & q\neq0\end{array}\right.$$ This spectral sequence degenerates at the second step and converges to $'E^{p+q}_{\infty}=\ext^{p+q}_{X}(\mathcal{G}, \mathcal{J})$.

The second spectral sequence has $E_{2}$-term:
\begin{align}
''E_{2}^{p,q}&=\h^{p}_{\delta}(C^{\bullet}(\mathcal{Y}, \mathcal{H}^{q}_{d}(\mathcal{H}om_{f}(\mathcal{G}, I^{\bullet})))) \notag\\
&=\h^{p}_{\delta}(C^{\bullet}(\mathcal{Y}, \mathcal{E}xt^{q}_{f}(\mathcal{G}, \mathcal{J}))) \notag
\end{align}
Then
\begin{equation*}
''E_{2}^{p,q}=\h^{p}(Y, \mathcal{E}xt^{q}_{f}(\mathcal{G}, \mathcal{J})).
\end{equation*}

\end{proof}

This proposition will be used in the next step to prove the existence of a display of a monad on $\tilde{\mathbb{P}}\times S$ associated to a family of framed torsion-free sheaves on $\tilde{\mathbb{P}}$. Consider the following extensions:
\begin{align}
&0\longrightarrow\mathcal{U}\longrightarrow\mathcal{K}\longrightarrow\mathcal{F}\longrightarrow0 \in\ext^{1}(\mathcal{F},\mathcal{U}) \\
&0\longrightarrow\mathcal{F}\longrightarrow\mathcal{Q}\longrightarrow\mathcal{V}\longrightarrow0 \in\ext^{1}(\mathcal{F},\mathcal{V})
\end{align}
where $\mathcal{U}=\oplus_{i=0}^{n}\mathcal{O}(-1,E_{i})\boxtimes\mathcal{U}_{i}$, with

$$\left\{\begin{array}{l}\mathcal{U}_{0}=\mathcal{R}^{1}\tilde{q}_{\ast}(\mathcal{F}\otimes\tilde{p}^{\ast}\mathcal{O}(-2,\vec{1})) \\ \mathcal{U}_{i}=\mathcal{R}^{1}\tilde{q}_{\ast}(\mathcal{F}\otimes\tilde{p}^{\ast}\mathcal{O}(-2,\vec{1}-E_{i})) \end{array}\right.$$
and $\mathcal{V}=\oplus_{i=0}^{n}\mathcal{O}(1,-E_{i})\boxtimes\mathcal{V}_{i}$ with

$$\left\{\begin{array}{l}\mathcal{V}_{0}=\mathcal{R}^{1}\tilde{q}_{\ast}(\mathcal{F}\otimes\tilde{p}^{\ast}\mathcal{O}(-2,0)) \\ \mathcal{V}_{i}=\mathcal{R}^{1}\tilde{q}_{\ast}(\mathcal{F}\otimes\tilde{p}^{\ast}\mathcal{O}(-2,\vec{1})). \end{array}\right.$$
We denoted by $\mathcal{R}^{i}\tilde{q}_{\ast}$ the i-th right derived functor of the direct image functor $\tilde{q}_{\ast}$ from the category of coherent sheaves of $\mathcal{O}_{\tilde{\mathbb{P}}\times S}$-modules to the category of coherent sheaves of $\mathcal{O}_{S}$-modules. We also remind the reader that we defined $E_{0}:=0.$ We remark that the extensions \eqref{A} and \eqref{B} exist since the groups $\ext^{1}(\mathcal{F},\mathcal{U})$ and $\ext^{1}(\mathcal{V},\mathcal{F})$ classifying them respectively are non trivial, as it will be proved in proposition ${\bf\ref{universal-existence}}$ below.

The sheaves $\mathcal{U}$ and $\mathcal{V}$ are locally free on $\tilde{\mathbb{P}}\times S$; let us consider, for the moment, the sheaves $\mathcal{U}_{0}$ and $\mathcal{U}_{i}$: by Grauert's theorem (\cite[\textbf{III. 12. 9}]{Hart}), for every point $s\in S$, there are maps:
\begin{align}
&\tilde{q}_{\ast}(\mathcal{F}\otimes\tilde{p}^{\ast}\mathcal{O}(-2,\vec{1}))\otimes k(s)\longrightarrow \h^{0}(\tilde{\mathbb{P}},\mathcal{E}(-2,\vec{1}))=0 \notag\\
&\tilde{q}_{\ast}(\mathcal{F}\otimes\tilde{p}^{\ast}\mathcal{O}(-2,\vec{1}-E_{i}))\otimes k(s)\longrightarrow \h^{0}(\tilde{\mathbb{P}},\mathcal{E}(-2,\vec{1}-E_{i}))=0 \notag\\
&\mathcal{R}^{2}\tilde{q}_{\ast}(\mathcal{F}\otimes\tilde{p}^{\ast}\mathcal{O}(-2,\vec{1}))\otimes k(s)\longrightarrow \h^{2}(\tilde{\mathbb{P}},\mathcal{E}(-2,\vec{1}))=0 \notag\\
&\mathcal{R}^{2}\tilde{q}_{\ast}(\mathcal{F}\otimes\tilde{p}^{\ast}\mathcal{O}(-2,\vec{1}-E_{i}))\otimes k(s)\longrightarrow \h^{2}(\tilde{\mathbb{P}},\mathcal{E}(-2,\vec{1}-E_{i}))=0 \notag.
\end{align}

Since all of these maps go to zero, the sheaves in the left-hand side have rank zero at all $s\in S$, thus they vanish identically (\cite[\textbf{III.12.5}]{Hart}). By the Riemann-Roch theorem, the dimensions of $\h^{1}(\tilde{\mathbb{P}},\mathcal{E}(-2,\vec{1}))$ and $\h^{1}(\tilde{\mathbb{P}},\mathcal{E}(-2,\vec{1}-E_{i}))$ are constant, hence $\mathcal{U}_{0}$ and $\mathcal{U}_{i}$ are locally free.
In the same way one can show that $\mathcal{V}_{0}$ and $\mathcal{V}_{i}$ are locally free.

\begin{pr}\label{universal-existence}
There exist non trivial extensions
\begin{align}
&0\longrightarrow\mathcal{U}\longrightarrow\mathcal{K}\longrightarrow\mathcal{F}\longrightarrow0 \in\ext^{1}(\mathcal{F},\mathcal{U}) \label{A}\\
&0\longrightarrow\mathcal{F}\longrightarrow\mathcal{Q}\longrightarrow\mathcal{V}\longrightarrow0 \in\ext^{1}(\mathcal{V},\mathcal{F}) \label{B}
\end{align}
\end{pr}

\begin{proof}
The statement of this proposition will be proved by showing that neither $\ext^{1}(\mathcal{F},\mathcal{U})$ nor $\ext^{1}(\mathcal{V},\mathcal{F})$ is trivial.

\vspace{0.5cm}

\underline{$\ext^{1}(\mathcal{V},\mathcal{F})$}:

\vspace{0.5cm}

One has $\ext^{1}(\mathcal{V},\mathcal{F})=\oplus_{i=0}^{n}\ext^{1}(\mathcal{O}(1,-E_{i})\boxtimes\mathcal{V}_{i}, \mathcal{F}).$ By the result of theorem \textbf{\ref{loctoglob}}, it follows that the terms contributing to each component $\ext^{1}(\mathcal{O}(1,-E_{i})\boxtimes\mathcal{V}_{i}, \mathcal{F})$ are given by $$\left\{\begin{array}{l} \h^{0}(S,\mathcal{E}xt_{\tilde{q}}^{1}(\tilde{q}^{\ast}\mathcal{V}_{i},\mathcal{F}\otimes \tilde{p}^{\ast}\mathcal{O}(-1,E_{i}))) \\  \h^{1}(S,\mathcal{H}om_{\tilde{q}}(\tilde{q}^{\ast}\mathcal{V}_{i},\mathcal{F}\otimes \tilde{p}^{\ast}\mathcal{O}(-1,E_{i}))) \end{array}\right.$$

The second term $\h^{1}(S,\mathcal{H}om_{\tilde{q}}(\tilde{q}^{\ast}\mathcal{V}_{i},\mathcal{F}\otimes \tilde{p}^{\ast}\mathcal{O}(-1,E_{i})))$ is trivial; this is because the sheaf $\mathcal{H}om_{\tilde{q}}(\tilde{q}^{\ast}\mathcal{V}_{i},\mathcal{F}\otimes \tilde{p}^{\ast}\mathcal{O}(-1,E_{i}))$ is identically zero. Indeed one has
\begin{align}
\mathcal{H}om_{\tilde{q}}(\tilde{q}^{\ast}\mathcal{V}_{i},\mathcal{F}\otimes \tilde{p}^{\ast}\mathcal{O}(-1,E_{i}))&= \tilde{q}_{\ast}[\mathcal{H}om(\tilde{q}^{\ast}\mathcal{V}_{i},\mathcal{F}\otimes \tilde{p}^{\ast}\mathcal{O}(-1,E_{i}))] \notag \\
&=\mathcal{V}^{\ast}_{i}\otimes \tilde{q}_{\ast}(\mathcal{F}\otimes \tilde{p}^{\ast}\mathcal{O}(-1,E_{i})).\notag
\end{align}
The second equality is obtained from the fact that $\mathcal{V}_{i}$ is locally free and by using the projection formula. Moreover, from the natural morphism $$\tilde{q}_{\ast}(\mathcal{F}\otimes \tilde{p}^{\ast}\mathcal{O}(-1,E_{i}))\otimes k(s)\longrightarrow\h^{0}(\tilde{\mathbb{P}},\mathcal{E}(-1,E_{i}))=0$$
at any closed point $s\in S,$ the claim follows.
\vspace{0.5cm}

Recall that the Grothendieck spectral sequence associated to the composition of the direct image functor $\tilde{q}_{\ast}:\mathcal{C}oh(\tilde{\mathbb{P}}\times S)\longrightarrow\mathcal{C}oh(S)$ with the local Hom functor $\mathcal{H}om(\mathcal{G},-):\mathcal{C}oh(\tilde{\mathbb{P}}\times S)\longrightarrow\mathcal{C}oh(\tilde{\mathbb{P}}\times S)$ for a fixed coherent sheaf $\mathcal{G}$ on $\tilde{\mathbb{P}}\times S$
has a second term $$E_{2}^{s,t}=\mathcal{R}^{t}\tilde{q}_{\ast}[\mathcal{E}xt^{s}(\mathcal{G},\mathcal{H})].$$ and converges to $E_{\infty}^{s,t}=\mathcal{E}xt_{\tilde{q}}^{s+t}(\mathcal{G},\mathcal{H}).$

It follows that the contributing terms to $\h^{0}(S,\mathcal{E}xt_{\tilde{q}}^{1}(\tilde{q}^{\ast}\mathcal{V}_{i},\mathcal{F}\otimes \tilde{p}^{\ast}\mathcal{O}(-1,E_{i}))),$ from the Grothendieck spectral sequence $E_{r}^{s,t}$ (we take the sheaf $\mathcal{G}$ to be $\tilde{q}^{\ast}\mathcal{V}_{i}$), are given by $$\left\{\begin{array}{l} \tilde{q}_{\ast}[\mathcal{E}xt^{1}(\tilde{q}^{\ast}\mathcal{V}_{i},\mathcal{F}\otimes \tilde{p}^{\ast}\mathcal{O}(-1,E_{i}))] \\ \mathcal{R}^{1}\tilde{q}_{\ast}[\mathcal{H}om(\tilde{q}^{\ast}\mathcal{V}_{i},\mathcal{F}\otimes \tilde{p}^{\ast}\mathcal{O}(-1,E_{i}))]  \end{array}\right.$$
Since $\mathcal{V}^{i}$ is locally free then $\mathcal{E}xt^{1}(\tilde{q}^{\ast}\mathcal{V}_{i},\mathcal{F}\otimes \tilde{p}^{\ast}\mathcal{O}(-1,E_{i}))=0.$

The other term can be written as
\begin{align}
\mathcal{R}^{1}\tilde{q}_{\ast}[\mathcal{H}om(\tilde{q}^{\ast}\mathcal{V}_{i},\mathcal{F}\otimes \tilde{p}^{\ast}\mathcal{O}(-1,E_{i}))]&=\mathcal{R}^{1}\tilde{q}_{\ast}[\tilde{q}^{\ast}\mathcal{V}^{\ast}_{i}\otimes\mathcal{F}\otimes \tilde{p}^{\ast}\mathcal{O}(-1,E_{i}))] \notag \\
&=\mathcal{V}^{\ast}_{i}\otimes\mathcal{R}^{1}\tilde{q}_{\ast}[\mathcal{F}\otimes \tilde{p}^{\ast}\mathcal{O}(-1,E_{i}))].\notag
\end{align}
The natural morphism, at a point $s\in S,$ given by $$\mathcal{R}^{1}\tilde{q}_{\ast}[\mathcal{F}\otimes \tilde{p}^{\ast}\mathcal{O}(-1,E_{i}))]\otimes k(s)\longrightarrow\h^{1}(\tilde{\mathbb{P}}, \mathcal{E}(-1,E_{i}))$$ shows that the fibre of the sheaf is not zero. By using the Riemann-Roch formula, it follows that the dimension of the space $\h^{1}(\tilde{\mathbb{P}}, \mathcal{E}(-1,E_{i}))$ is non zero and independent from the point $s,$ since $\mathcal{F}$ is $S-$flat and $\h^{0,2}(\tilde{\mathbb{P}}, \mathcal{E}(-1,E_{i}))=0.$

Thus, each component of the extension $\ext^{1}(\mathcal{V},\mathcal{F})$ is given by $$\ext^{1}(\mathcal{O}(-1,E_{i})\boxtimes\mathcal{V}_{i}, \mathcal{F})=\h^{0}(\tilde{\mathbb{P}}\times S,\mathcal{H}om(\mathcal{V}_{i},\mathcal{R}^{1}\tilde{q}_{\ast}[\mathcal{F}\otimes \tilde{p}^{\ast}\mathcal{O}(-1,E_{i})]))$$

Now for $i=0$ one has $\mathcal{H}om(\mathcal{V}_{0},\mathcal{R}^{1}\tilde{q}_{\ast}[\mathcal{F}\otimes \tilde{p}^{\ast}\mathcal{O}(-1,0)])=\mathcal{E}nd(\mathcal{V}_{0}).$ Moreover we have a surjective morphism $\ext^{1}(\mathcal{V},\mathcal{F})\to \mathcal{E}nd(\mathcal{V}_{0}),$ which is the projection on the zero-th factor. Then there exist at least one non-trivial extension which is sent to the identity element $1\in\End(\mathcal{V}_{0}).$ Hence at least one non-trivial extension \eqref{B} exists.

\vspace{0.5cm}

\underline{$\ext^{1}(\mathcal{F},\mathcal{U})$}:

\vspace{0.5cm}

One has $\ext^{1}(\mathcal{F},\mathcal{U})=\oplus_{i=0}^{n}\ext^{1}(\mathcal{F},\mathcal{U}_{i}\boxtimes \tilde{p}^{\ast}\mathcal{O}(1,-E_{i})).$ By using the spectral sequence in \textbf{\ref{loctoglob}} one shows that $$\ext^{1}(\mathcal{F},\mathcal{U}_{i}\boxtimes \tilde{p}^{\ast}\mathcal{O}(1,-E_{i}))=\h^{0}(S,\mathcal{E}xt_{\tilde{q}}^{1}(\mathcal{F}\otimes \tilde{p}^{\ast}\mathcal{O}(1,-E_{i}),\tilde{q}^{\ast}\mathcal{U}_{i}))$$ since the sheaf $\mathcal{H}om_{\tilde{q}}(\mathcal{F}\otimes \tilde{p}^{\ast}\mathcal{O}(1,-E_{i}),\tilde{q}^{\ast}\mathcal{U}_{i})$ vanishes identically; indeed one has
$$\mathcal{H}om_{\tilde{q}}(\mathcal{F}\otimes \tilde{p}^{\ast}\mathcal{O}(1,-E_{i}),\tilde{q}^{\ast}\mathcal{U}_{i})\otimes k(s)\to\Hom(\mathcal{E}(1,-E_{i}),\h^{1}(\tilde{\mathbb{P}},\mathcal{E}(-2,\vec{1}-E_{i})))$$
but $$\Hom(\mathcal{E}(1,-E_{i}),\h^{1}(\tilde{\mathbb{P}},\mathcal{E}(-2,\vec{1}-E_{i})))=\h^{1}(\tilde{\mathbb{P}},\mathcal{E}(-2,\vec{1}-E_{i})) \otimes\underbrace{\h^{2}(\tilde{\mathbb{P}},\mathcal{E}(-2,\vec{1}-E_{i}))^{\ast}}_{0}=0.$$

On the other hand we have $\mathcal{E}xt_{\tilde{q}}^{1}(\mathcal{F}\otimes \tilde{p}^{\ast}\mathcal{O}(1,-E_{i}),\tilde{q}^{\ast}\mathcal{U}_{i})=\mathcal{E}xt_{\tilde{q}}^{1}(\mathcal{F}\otimes \tilde{p}^{\ast}\mathcal{O}(-2,\vec{1}-E_{i}),\tilde{p}^{\ast}\mathcal{O}(-3,\vec{1})\otimes \tilde{q}^{\ast}\mathcal{U}_{i}).$
By using the relative duality morphism, \cite[Main Theorem]{Kleiman}, one has $$\mathcal{E}xt_{\tilde{q}}^{1}(\mathcal{F}\otimes \tilde{p}^{\ast}\mathcal{O}(1,-E_{i}),\tilde{q}^{\ast}\mathcal{U}_{i})\cong \mathcal{H}om_{S}(\mathcal{R}^{1}\tilde{q}_{\ast}(\mathcal{F}\otimes \tilde{p}^{\ast}\mathcal{O}(-2,\vec{1}-E_{i})),\mathcal{U}_{i})=\mathcal{E}nd(\mathcal{U}_{i}).$$
Notice that we used the fact that the relative dualizing sheaf is isomorphic to $\tilde{p}^{\ast}\mathcal{O}(-3,\vec{1}).$
Thus $\ext^{1}(\mathcal{F},\mathcal{U})=\oplus_{i=0}^{n}\End(\mathcal{U}_{i}).$ Moreover the morphism $\ext^{1}(\mathcal{F},\mathcal{U})\to\End(\mathcal{U}_{i})$ is surjective, $\forall i=0,\cdots,n,$ since it is just the natural projection on the $i-$th factor. Hence for every $i$ there exists at least one non-trivial projection which is sent to $1\in\End(\mathcal{U}_{i}).$ It follows that $\ext^{1}(\mathcal{F},\mathcal{U})$ is not trivial.

\end{proof}

\begin{cor}
There exist two extensions
\begin{align}
&0\longrightarrow\mathcal{U}\longrightarrow\mathcal{W}\longrightarrow\mathcal{Q}\longrightarrow0 \label{C}\\
&0\longrightarrow\mathcal{K}\longrightarrow\mathcal{W}\longrightarrow\mathcal{V}\longrightarrow0 \label{D}
\end{align}
such that the extension \eqref{A} and \eqref{B} fit in the following display of a monad
\begin{equation}
\xymatrix@R-1pc@C-1pc{&0\ar[d]&0\ar[d]&& \\
&\mathcal{U}\ar[d]\ar@{=}[r]&\mathcal{U}\ar[d]&& \\
0\ar[r]&\mathcal{K}\ar[d]\ar[r]&\mathcal{W}\ar[d]\ar[r]&\mathcal{V}\ar@{=}[d]\ar[r]&0 \\
0\ar[r]&\mathcal{F}\ar[d]\ar[r]&\mathcal{Q}\ar[d]\ar[r]&\mathcal{V}\ar[r]&0\\
&0&0&&
}
\end{equation}
\end{cor}

\begin{proof}
By proposition \textbf{\ref{King223}}, the two extensions, \eqref{A} and \eqref{B}, fit into the display of a monad on $\tilde{\mathbb{P}}\times S$ if and only if their $\ext$-product in $\ext^{2}(\mathcal{V},\mathcal{U})$ is trivial. To show this vanishing property we use the "relative local to global" spectral sequence that we constructed above; we have $$E_{2}^{p,q}=\h^{p}(S, \mathcal{E}xt^{\tilde{q}}_{\tilde{\tilde{q}}}(\mathcal{V}, \mathcal{U}))\Longrightarrow \ext^{p+q}_{\tilde{\mathbb{P}}\times S}(\mathcal{V}, \mathcal{U})$$

The spectral sequence terms which contributes to $\ext^{2}(\mathcal{V},\mathcal{U})$ are $$\h^{0}(S, \mathcal{E}xt^{2}_{\tilde{\tilde{q}}}(\mathcal{V}, \mathcal{U})),\quad \h^{1}(S, \mathcal{E}xt^{1}_{\tilde{q}}(\mathcal{V}, \mathcal{U}))\quad \textrm{ and }\quad \h^{2}(S, \mathcal{H}om_{\tilde{q}}(\mathcal{V}, \mathcal{U})).$$ Next we will prove that both $\h^{2}(S, \mathcal{H}om_{\tilde{q}}(\mathcal{V}, \mathcal{U}))$ and $\h^{0}(S, \mathcal{E}xt^{2}_{\tilde{q}}(\mathcal{V}, \mathcal{U}))$ vanish, hence $\ext^{2}(\mathcal{V},\mathcal{U})=\h^{1}(S, \mathcal{E}xt^{1}_{\tilde{q}}(\mathcal{V}, \mathcal{U}))$.

\bigskip

\underline{$\h^{2}(S, \mathcal{H}om_{\tilde{q}}(\mathcal{V}, \mathcal{U}))$}:

$$\mathcal{H}om_{\tilde{q}}(\mathcal{V}, \mathcal{U})\otimes k(s)=\oplus_{i,j=0}^{n}\tilde{q}_{\ast}[\mathcal{H}om(\tilde{q}^{\ast}\mathcal{V}_{i}, \tilde{q}^{\ast}\mathcal{U}_{j}\otimes\tilde{p}^{\ast}\mathcal{O}(-2,E_{i}+E_{j}))]\otimes k(s),$$
There is a natural map
$$\oplus_{i,j=0}^{n}\tilde{q}_{\ast}[\mathcal{H}om(\tilde{q}^{\ast}\mathcal{V}_{i}, \tilde{q}^{\ast}\mathcal{U}_{j}\otimes\tilde{p}^{\ast}\mathcal{O}(-2,E_{i}+E_{j}))]\otimes k(s)\qquad\qquad\qquad$$
$$\qquad\qquad\qquad\qquad\qquad\longrightarrow \oplus_{i,j=0}^{n}\Hom(\tilde{q}^{\ast}\mathcal{V}_{i}(s), \tilde{q}^{\ast}\mathcal{U}_{j}(s)\otimes\mathcal{O}(-2,E_{i}+E_{j}))$$ which is an isomorphism if it is surjective. Then the pull-back $\tilde{q}^{\ast}\mathcal{V}_{i}(s)$, of the stalk $\mathcal{V}_{i}(s)$ at a point $s$, is just the pull-back associated to the map $\tilde{\mathbb{P}}\longrightarrow s=\spec k(s)$, hence $\tilde{q}^{\ast}\mathcal{V}_{i}(s)$ is constant on $\tilde{\mathbb{P}}$. This is also true for the pull-back $\tilde{q}^{\ast}\mathcal{U}_{j}(s)$ of $\mathcal{U}_{j}(s)$. Thus
$$\oplus_{i,j=0}^{n}\Hom(\tilde{q}^{\ast}\mathcal{V}_{i}(s), \tilde{q}^{\ast}\mathcal{U}_{j}(s)\otimes\mathcal{O}(-2,E_{i}+E_{j}))= \oplus_{i,j=0}^{n} [\tilde{q}^{\ast}\mathcal{V}_{j}(s)]^{\ast}\otimes\tilde{q}^{\ast}\mathcal{U}_{j}(s)\otimes \underbrace{\h^{0}(\tilde{\mathbb{P}},\mathcal{O}(-2,E_{i}+E_{j}))}_{0}$$
every $i,j=0,...,n,$ and the above natural map is obviously surjective. Furthermore this implies the vanishing of $\mathcal{H}om_{\tilde{q}}(\mathcal{V}, \mathcal{U})$ identically. Hence $\h^{2}(S, \mathcal{H}om_{\tilde{q}}(\mathcal{V}, \mathcal{U}))=0.$

\bigskip

\underline{$\h^{0}(S, \mathcal{E}xt^{2}_{\tilde{q}}(\mathcal{V}, \mathcal{U}))$}: By the relative Serre duality \cite{Kleiman} we have
$$\mathcal{E}xt^{2}_{\tilde{q}}(\mathcal{V}, \mathcal{U})\otimes k(s)=\oplus_{i,j=0}^{n}\{\tilde{q}_{\ast}[\mathcal{H}om(\tilde{q}^{\ast}\mathcal{U}_{j}, \tilde{q}^{\ast}\mathcal{V}_{i}\otimes\tilde{p}^{\ast}\mathcal{O}(-1,\vec{1}-E_{i}-E_{j}))]\}^{\ast}\otimes k(s)$$

In the same way as above, there is  a map
$$\oplus_{i,j=0}^{n}\{\tilde{q}_{\ast}[\mathcal{H}om(\tilde{q}^{\ast}\mathcal{U}_{j}, \tilde{q}^{\ast}\mathcal{V}_{i}\otimes\tilde{p}^{\ast}\mathcal{O}(-1,\vec{1}-E_{i}-E_{j}))]\}^{\ast}\otimes k(s)\qquad\qquad\qquad\qquad$$ $$\qquad\qquad\qquad\qquad\longrightarrow \oplus_{i,j=0}^{n}\Hom(\tilde{q}^{\ast}\mathcal{U}_{j}(s), \tilde{q}^{\ast}\mathcal{V}_{i}(s)\otimes\mathcal{O}(-1,\vec{1}-E_{i}-E_{j}))^{\ast}$$
which is an isomorphism if it is surjective. This map is surjective since for every $i,j=0,...,n$ in the direct sum we have
$$\Hom(\tilde{q}^{\ast}\mathcal{U}_{j}(s), \tilde{q}^{\ast}\mathcal{V}_{i}(s)\otimes\mathcal{O}(-1,\vec{1}-E_{i}-E_{j}))^{\ast}= \tilde{q}^{\ast}\mathcal{U}_{j}(s)\otimes[\tilde{q}^{\ast}\mathcal{V}_{i}(s)]^{\ast}\otimes \underbrace{\h^{0}(\tilde{\mathbb{P}},\mathcal{O}(-1,\vec{1}-E_{i}-E_{j}))^{\ast}}_{0}$$ which implies that $\mathcal{E}xt^{2}_{\tilde{q}}(\mathcal{V}, \mathcal{U})=0$ and hence $\h^{0}(S, \mathcal{E}xt^{2}_{\tilde{q}}(\mathcal{V}, \mathcal{U}))=0.$

\bigskip

Computing the fiber of the sheaf $\mathcal{E}xt^{1}_{\tilde{q}}(\mathcal{V}, \mathcal{U})$, over a point $s\in S$, one can check that it is not zero, hence $\ext^{2}(\mathcal{V},\mathcal{U})=\h^{1}(S, \mathcal{E}xt^{1}_{\tilde{q}}(\mathcal{V}, \mathcal{U}))$ may not vanish in general. However one can overcome this obstruction as follows: the sequences \eqref{A} and \eqref{B} are classified by $\ext^{1}(\mathcal{F}, \mathcal{U})$ and $\ext^{1}(\mathcal{V}, \mathcal{F}),$ respectively, and one can compute their Yoneda product. If this product is zero then one can apply proposition \textbf{\ref{King223}} to construct a display of a monad. So we shall prove that the Yoneda product
\begin{equation}\label{pairing}
\ext^{1}(\mathcal{V}, \mathcal{F})\times\ext^{1}(\mathcal{F}, \mathcal{U})\longrightarrow\ext^{2}(\mathcal{V}, \mathcal{U})
\end{equation}
is zero.

As for the sheaf $\mathcal{H}om_{\tilde{q}}(\mathcal{V}, \mathcal{U})$, one can prove that the sheaves $\mathcal{H}om_{\tilde{q}}(\mathcal{F}, \mathcal{U})$ and $\mathcal{H}om_{\tilde{q}}(\mathcal{V}, \mathcal{F})$ vanish identically by computing their associated fibers at a given point $s\in S.$ By the spectral sequence in proposition \textbf{\ref{loctoglob}}, it follows that $$\ext^{1}(\mathcal{F}, \mathcal{U})=\h^{0}(S,\mathcal{E}xt^{1}_{\tilde{q}}(\mathcal{F}, \mathcal{U}))\quad\textrm{ and }\quad\ext^{1}(\mathcal{V}, \mathcal{F})=\h^{0}(S, \mathcal{E}xt^{1}_{\tilde{q}}(\mathcal{V}, \mathcal{F})),$$
and the pairing \eqref{pairing} can be written as

$$\h^{0}(S,\mathcal{E}xt^{1}_{\tilde{q}}(\mathcal{V}, \mathcal{F})\times\mathcal{E}xt^{1}_{\tilde{q}}(\mathcal{F}, \mathcal{U}))\longrightarrow \ext^{2}(\mathcal{V}, \mathcal{U}).$$

On the other hand this product is inherited from the relative local version of the Yoneda product \cite[Section \textbf{10.1.7}]{Huy}
$$\mathcal{E}xt^{1}_{\tilde{q}}(\mathcal{V}, \mathcal{F})\times\mathcal{E}xt^{1}_{\tilde{q}}(\mathcal{F}, \mathcal{U})\longrightarrow \mathcal{E}xt^{2}_{\tilde{q}}(\mathcal{V}, \mathcal{U})$$
Thus the global Yoneda product takes values in the $\h^{0}(S,\mathcal{E}xt^{2}_{\tilde{q}}(\mathcal{V}, \mathcal{U}))$ part of the group $\ext^{2}(\mathcal{V}, \mathcal{U})$. But we already proved that $\h^{0}(S,\mathcal{E}xt^{2}_{\tilde{q}}(\mathcal{V}, \mathcal{U}))=0$, consequently the Yoneda product of the extensions \eqref{A} and \eqref{B} is zero, implying the existence of sequences

\begin{align}
&0\longrightarrow\mathcal{U}\longrightarrow\mathcal{W}\longrightarrow\mathcal{Q}\longrightarrow0 \notag\\
&0\longrightarrow\mathcal{K}\longrightarrow\mathcal{W}\longrightarrow\mathcal{V}\longrightarrow0 \notag
\end{align}
such that the diagram
\begin{equation}
\xymatrix@R-1pc@C-1pc{&0\ar[d]&0\ar[d]&& \\
&\mathcal{U}\ar[d]\ar@{=}[r]&\mathcal{U}\ar[d]&& \\
0\ar[r]&\mathcal{K}\ar[d]\ar[r]&\mathcal{W}\ar[d]\ar[r]&\mathcal{V}\ar@{=}[d]\ar[r]&0 \\
0\ar[r]&\mathcal{F}\ar[d]\ar[r]&\mathcal{Q}\ar[d]\ar[r]&\mathcal{V}\ar[r]&0\\
&0&0&&
}
\end{equation}
commutes.
\end{proof}

The display existence shows that there exists a monad
\begin{equation}\label{S-monad}
\mathcal{M}: \quad\oplus_{i=0}^{n}\mathcal{O}_{\tilde{\mathbb{P}}}(-1,E_{i})\boxtimes\mathcal{U}_{i}\longrightarrow \mathcal{W} \longrightarrow\oplus_{i=0}^{n}\mathcal{O}_{\tilde{\mathbb{P}}}(1,-E_{i})\boxtimes\mathcal{V}_{i}
\end{equation}
associated to the family $\mathcal{F}$ on $\tilde{\mathbb{P}}$ parameterized by $S$. One can show that the restriction to the fibers of $\tilde{q}$ gives a monad isomorphic to the one in Proposition \textbf{2.5} and, using the display, one can show that the second term $\mathcal{W}$ of the monad $\mathcal{M}$ is trivial along the fibers of $\tilde{q}.$

\bigskip

Another useful monad we now introduce is the universal monad on $\tilde{\mathbb{P}}\times P$ :
\begin{equation}\label{univ-monad}
\xymatrix@C-1pc{\mathbb{M}:& \oplus_{i=0}^{n}\mathcal{O}_{\tilde{\mathbb{P}}}(-1,E_{i})\boxtimes K_{i}\otimes\mathcal{O}_{P}\ar[r]&  \mathcal{O}_{\tilde{\mathbb{P}}}\boxtimes W\otimes\mathcal{O}_{P}\ar[r]& \oplus_{i=0}^{n}\mathcal{O}_{\tilde{\mathbb{P}}}(1,-E_{i})\boxtimes L_{i}\otimes\mathcal{O}_{P}}
\end{equation}
We denote its cohomology by $\mathfrak{F}.$ The vector spaces $K_{i}$ and $L_{i}$ are given by Proposition \textbf{\ref{prop}}. Note that at each point $p$ of the space $P,$ defined in \eqref{spaceP}, the maps $\alpha_{p}$ and $\beta_{p},$ as in \eqref{alpha} and \eqref{beta}, gives a monad $M(p)$ whose cohomology, $\mathcal{E}(p),$ is a framed torsion-free sheaf. By varying the point $p$ in on an open neighborhood $U\subset P,$ the maps $\alpha$ and $\beta$ will depend on the sheaf $\mathcal{O}_{U}.$ The action of the group $G,$ defined in \eqref{P and G}, provides the gluing of the maps $\alpha$ and $\beta$ in the intersection of two open sets. This allows the construction of the canonical monad \eqref{univ-monad} on $\tilde{\mathbb{P}}\times P$. By construction, the Chern character of cohomology $\mathfrak{F}\otimes k(p)\cong\mathcal{E}(p)$ is independent of the point $p\in P.$ Hence, by the Riemann-Roch formula, the Hilbert polynomial is independent from $p.$ Thus $\mathfrak{F}$ is flat over $P.$

\bigskip

\section{Smoothness of the space $P$}

A needed step, in proving smoothness of the moduli space, is to show the smoothness of the space $P$ of ADHM type configurations subject to the monad condition and defined in \eqref{spaceP}. In this section, we follow the argument given by Okonek et. al. in \cite[Chapter\textbf{II}, \textbf{4.1}]{Okonek}, provided that one makes suitable changes in order for the proofs to work in the coherent case. We need first to prove the following
\begin{thm}\label{Kuneth}
Let $L_{\bullet}$ and $M^{\bullet}$ be two bounded complexes of coherent sheaves of $\mathcal{O}_{X}$-modules over an algebraic variety $X$, and suppose that $L_{i}$ is locally-free or $M^{i}$ is injective, for every $i$ in the complexes. Then there is a spectral sequence $E_{r}^{pq}$ with second term $$E_{2}^{pq}=\oplus_{i+j=q}\mathcal{E}xt^{p}(H_{i}(L_{\bullet}),H^{j}(M^{\bullet}))$$ for which the $E_{\infty}$-term is the bi-graded group associated to a suitable filtration of the cohomology of the complex $\mathcal{H}om(L_{\bullet},M^{\bullet}).$
\end{thm}

\begin{proof}
The proof is similar to \cite[Theorem \textbf{5.4.1}]{Godement}, replacing projectives by locally-free.
\end{proof}

\begin{lem}\label{obstruction}
On a non-singular algebraic surface $X,$ let $\xymatrix{M:0\ar[r]&A\ar[r]^{a_{0}}&B\ar[r]^{b_{0}}&C\ar[r]&0}$ be a monad whose cohomology is a torsion-free sheaf $\mathcal{E}.$ Let $d_{0}$ be the following homomorphism $$\xymatrix@R-1pc@C-1pc{d_{0}:&\Hom(A,B)\oplus\Hom(B,C)\ar[r]&\Hom(A,C)\\ &(a,b)\ar[r]&d_{0}(a,b)=ba_{0}+b_{0}a}.$$ Then the map $$\h^{2}(X, \mathcal{E}nd(\mathcal{E}))\longrightarrow Coker d_{0}$$ is surjective if the following groups vanish:
$$\xymatrix@R-1.6pc@C-2pc{\h^{1}(X,\mathcal{H}om(B,C)),&\h^{1}(X,\mathcal{H}om(A,B)),&\h^{1}(X,\mathcal{E}nd(B)),&\h^{1}(X,\mathcal{E}nd(C))\\
\h^{1}(X,\mathcal{E}nd(A))\quad,&\h^{2}(X,\mathcal{E}nd(A))\quad,&\h^{2}(X,\mathcal{E}nd(B))\quad,&\h^{2}(X,\mathcal{E}nd(C))\\
\h^{2}(X,\mathcal{H}om(B,A)),&\h^{2}(X,\mathcal{H}om(C,B))&&}$$
\end{lem}

\begin{proof}
The proof is a generalization of \cite[Chapter \textbf{II}, \textbf{4.1}]{Okonek} replacing tesor product in the case of bundles by suitable sheaves of local homomorphisms in the coherent case. Consider the double complex
\begin{displaymath}
\xymatrix@R-1.4pc@C-1pc{\mathcal{H}om(C,A)\ar[r]\ar[d]&\mathcal{H}om(C,B)\ar[r]\ar[d]&\mathcal{E}nd(C)\ar[d]\\
\mathcal{H}om(B,A)\ar[r]\ar[d]&\mathcal{E}nd(B)\ar[r]\ar[d]&\mathcal{H}om(B,C)\ar[d] \\
\mathcal{E}nd(A)\ar[r]&\mathcal{H}om(A,B)\ar[r]&\mathcal{H}om(A,C)
}
\end{displaymath}
The associated total complex is given by:
\begin{displaymath}
\xymatrix@C-1pc{K^{\bullet}:&0\ar[r]&K^{-2}\ar[r]&K^{-1}\ar[r]^S&K^{0}\ar[r]^T&K^{1}\ar[r]^U&K^{2}\ar[r]&0}
\end{displaymath}
where
\begin{displaymath}
\xymatrix@R-2pc@C-2pc{K^{-2}=\mathcal{H}om(C,A),& K^{-1}=\mathcal{H}om(C,B)\oplus\mathcal{H}om(B,A),\\
K^{0}=\mathcal{E}nd(A)\oplus\mathcal{E}nd(B)\oplus\mathcal{E}nd(C), &K^{1}=\mathcal{H}om(A,B)\oplus\mathcal{H}om(B,C), \\
K^{2}=\mathcal{H}om(A,C).&
}
\end{displaymath}

Using theorem \textbf{\ref{Kuneth}} one has $$E_{2}^{pq}=\oplus_{i+j=q}\mathcal{E}xt^{p}(\h_{i}(L_{\bullet}),\h^{j}(M^{\bullet}))\Longrightarrow E_{\infty}^{p+q}=\h^{p+q}(\mathcal{H}om(L_{\bullet},M^{\bullet})).$$ This allows one to compute the cohomology of the complex $K^{\bullet}$:
\begin{equation}
\xymatrix@R-1.6pc@C-1pc{\h^{-2}(K^{\bullet})=0,&\h^{-1}(K^{\bullet})=0,&\h^{2}(K^{\bullet})=0, \\
\h^{0}(K^{\bullet})=\mathcal{E}nd(\mathcal{E}),&\h^{1}(K^{\bullet})=\mathcal{E}xt^{1}(\mathcal{E},\mathcal{E}).&
}
\end{equation}
So one can write the following exact sequences:
\begin{align}
&\xymatrix@C-1pc{0\ar[r]&ImS\ar[r]&KerT\ar[r]&\mathcal{E}nd(\mathcal{E})\ar[r]&0}\label{S1} \\
&\xymatrix@C-1pc{0\ar[r]&KerT\ar[r]&K^{0}\ar[r]&ImT\ar[r]&0}\label{S2} \\
&\xymatrix@C-1pc{0\ar[r]&K^{-2}\ar[r]&K^{-1}\ar[r]&ImS\ar[r]&0}\label{S3} \\
&\xymatrix@C-1pc{0\ar[r]&KerU\ar[r]&K^{1}\ar[r]&K^{2}\ar[r]&0}\label{S4} \\
&\xymatrix@C-1pc{0\ar[r]&ImT\ar[r]&KerU\ar[r]& \mathcal{E}xt^{1}(\mathcal{E},\mathcal{E})\ar[r]&0}\label{S5}
\end{align}

Now $d_{0}$ in the lemma is the map $d_{0}:\h^{0}(X,K^{1})\longrightarrow\h^{0}(X,K^{2})$ where
\begin{displaymath}
\xymatrix@C-1pc{\h^{0}(X,K^{1})=\h^{0}(X,\mathcal{H}om(A,B)\oplus\mathcal{H}om(B,C)),&\h^{0}(X,K^{2})=\mathcal{H}om(A,C).}
\end{displaymath}
From \eqref{S1}, \eqref{S2}, \eqref{S3} and the vanishing of the cohomology groups in the assumptions of the lemma, one has
\begin{align}
\h^{2}(X,\mathcal{E}nd(\mathcal{E}))&=\h^{2}(X,KerT)\notag \\
&=\h^{1}(X,ImT).\notag
\end{align}
From \eqref{S4} one has $Cokerd_{0}=\h^{1}(X,KerU).$ Combining all this with the long sequence in cohomology induced by \eqref{S5} one has
\begin{equation}
\cdots\longrightarrow\h^{2}(X,\mathcal{E}nd(\mathcal{E}))\longrightarrow Cokerd_{0}\longrightarrow\h^{1}(X,\mathcal{E}xt^{1}(\mathcal{E},\mathcal{E}))
\end{equation}

To prove that $\h^{1}(X,\mathcal{E}xt^{1}(\mathcal{E},\mathcal{E}))=0$ one uses first the natural sequence \eqref{doubledual}. Applying $\mathcal{H}om(\bullet,\mathcal{E})$ to \eqref{doubledual} it follows that $\mathcal{E}xt^{1}(\mathcal{E},\mathcal{E})=\mathcal{E}xt^{2}(\Delta,\mathcal{E})$ and $\mathcal{E}xt^{2}(\mathcal{E},\mathcal{E})=0$

The restriction to any open subset $V\subset X$ gives $\mathcal{E}xt^{2}(\Delta,\mathcal{E})|_{V}=\ext^{2}(\Delta|_{V},\mathcal{E}|_{V})$ but $\Delta|_{V}=0$ for $Supp\Delta\cap V=\emptyset.$ Hence $$\mathcal{E}xt^{2}(\Delta,\mathcal{E})|_{V}=\left\{\begin{array}{ll}\ext^{2}(\Delta|{V},\mathcal{E}|_{V}) & Supp\Delta\cap V\neq\emptyset.\\ 0 &Supp\Delta\cap V=\emptyset.\end{array}\right.$$
It follows that $\mathcal{E}xt^{2}(\Delta,\mathcal{E})$ is also supported on points, namely $Supp\Delta.$ Thus $\mathcal{E}xt^{1}(\mathcal{E},\mathcal{E})$ is supported on $\Delta.$ Consequently $\h^{1}(X,\mathcal{E}xt^{1}(\mathcal{E},\mathcal{E}))=0.$
\end{proof}

\bigskip

\begin{pr}
The space $P,$ of ADHM type configurations, is smooth.
\end{pr}

\begin{proof}
Consider the following monad $$\xymatrix{M:&0\ar[r]&\oplus_{i=0}^{n}K_{i}(-1,E_{i})\ar[r]^{\quad\quad\alpha}& W\ar[r]^{\beta\quad\quad}& \oplus_{i=0}^{n}L_{i}(1,-E_{i})\ar[r]&0}$$ on $\tilde{\mathbb{P}}$ and the mapping
\begin{displaymath}
\xymatrix@R-1pc@C-1pc{g:&\mathbb{H}\oplus\mathbb{F}\ar[r]&\mathbb{G} \\
&(\alpha,\beta)\ar[r]&\beta\circ\alpha
}
\end{displaymath}

The spaces $\mathbb{H}$, $\mathbb{F}$ and $\mathbb{G}$ are defined by
$$\left\{\begin{array}{l}\mathbb{H}=\oplus_{i=0}^{n}\Hom(V_{i}, \Hom(K_{i},W))\\ \mathbb{F}=\oplus_{i=0}^{n}\Hom(V_{i}, \Hom(W,L_{i})) \\ \mathbb{G}=\oplus_{i,j=0}^{n}\Hom(V_{ij}, \Hom(K_{i}),L_{j}))\end{array}\right.$$ The spaces $V_{i}$ and $V_{ij}$ are defined so that $V_{i}^{\ast}=\h^{0}(\tilde{\mathbb{P}},\mathcal{O}(1,-E_{i})),$ $V_{ij}^{\ast}=\h^{0}(\tilde{\mathbb{P}},\mathcal{O}(2,-E_{i}-E_{j})).$
The differential of $g$ at a point $(\alpha_{0},\beta_{0})$ is $$dg_{(\alpha_{0},\beta_{0})}=\beta_{0}\circ\alpha+\beta\circ\alpha_{0}.$$ Since $P\subset g^{-1}(0)$ is a Zariski-open subset, it suffices to show that $dg_{(\alpha_{0},\beta_{0})}$ is surjective for any point $(\alpha_{0},\beta_{0}).$ Using lemma \textbf{\ref{obstruction}} and verifying the vanishing of its hypothesis in our case, it follows that $$\h^{2}(\tilde{\mathbb{P}},\mathcal{E}nd(\mathcal{E}))\longrightarrow Coker(dg_{(\alpha_{0},\beta_{0})})\longrightarrow0$$ is exact. Finally we show that $\h^{2}(\tilde{\mathbb{P}},\mathcal{E}nd(\mathcal{E}))=0:$ from the local to global spectral sequence \cite[Section \textbf{II.7.4}]{Godement} one has
\begin{displaymath}
E_{2}^{pq}=\h^{p}(\tilde{\mathbb{P}},\mathcal{E}xt^{q}(\mathcal{E},\mathcal{E}))\Longrightarrow\ext^{p+q}(\mathcal{E},\mathcal{E})
\end{displaymath}
and since $\h^{1}(\tilde{\mathbb{P}},\mathcal{E}xt^{1}(\mathcal{E},\mathcal{E}))=0$ and $\mathcal{E}xt^{2}(\mathcal{E},\mathcal{E})$ is the zero sheaf, then we have $\h^{2}(\mathcal{E}nd(\mathcal{E},\mathcal{E}))=\ext^{2}(\mathcal{E},\mathcal{E}).$ On the other hand $\ext^{2}(\mathcal{E},\mathcal{E})=\Hom(\mathcal{E},\mathcal{E}(-3,\overrightarrow{1}))^{\ast}.$ By using
\begin{displaymath}
\xymatrix@C-1pc{0\ar[r]&\mathcal{E}(-3,\overrightarrow{1})\ar[r]&\mathcal{E}(-2,\overrightarrow{1})\ar[r]&\mathcal{E}|_{l_{\infty}}(-2)\ar[r]&0}
\end{displaymath}
it follows
\begin{displaymath}
\xymatrix@C-1pc{0\ar[r]&\Hom(\mathcal{E},\mathcal{E}(-3,\overrightarrow{1}))\ar[r]&\Hom(\mathcal{E},\mathcal{E}(-2,\overrightarrow{1})) \ar[r]&\Hom(\mathcal{E},\mathcal{E}|_{l_{\infty}}(-2))\ar[r]&0}
\end{displaymath}
But $\Hom(\mathcal{E},\mathcal{E}(-2,\overrightarrow{1}))=\ext^{2}(\mathcal{E},\mathcal{E}(-1,0))^{\ast}$ which also vanish because of the following; twisting the sequence by $\mathcal{O}(-1,0)$, and applying the functor $\Hom(\mathcal{E}^{\ast\ast},\bullet)$ one gets the long exact sequence:
$$\longrightarrow \ext^{1}(\mathcal{E}^{\ast\ast},\Delta))\longrightarrow \ext^{2}(\mathcal{E}^{\ast\ast},\mathcal{E}(-1,0))\longrightarrow \ext^{2}(\mathcal{E}^{\ast\ast},\mathcal{E}^{\ast\ast}(-1,0))\longrightarrow0.$$
But $\ext^{1}(\mathcal{E}^{\ast\ast},\Delta))\cong\h^{1}(\Sigma,\Delta))^{\oplus r}$ vanishes since $Supp(\Delta)$ is zero dimensional. $\ext^{2}(\mathcal{E}^{\ast\ast},\mathcal{E}^{\ast\ast}(-1,0))$ $=\h^{2}(\tilde{\mathbb{P}}, \mathcal{E}^{\ast}\otimes\mathcal{E}^{\ast\ast}(-1,0))$ vanishes since for a framed sheaf $\mathcal{F}$,  $\h^{2}(\tilde{\mathbb{P}}, \mathcal{F}(-1,\vec{q})=0 \quad\forall \vec{q}$. Hence $\ext^{2}(\mathcal{E}^{\ast\ast},\mathcal{E}(-1,0))=0$.

Applying again the functor $\Hom(\bullet,\mathcal{E}(-1,0))$ on the sequence \eqref{doubledual}, one has
$$\longrightarrow \ext^{2}(\mathcal{E}^{\ast\ast},\mathcal{E}(-1,0))\longrightarrow \ext^{2}(\mathcal{E},\mathcal{E}(-1,0))\longrightarrow0.$$

Thus $\ext^{2}(\mathcal{E},\mathcal{E}(-1,0))=0$

\end{proof}

\bigskip

\section{Smoothness of the moduli space $\mathcal{M}^{\tilde{\mathbb{P}}}_{\vec{a}, k}$}
In this section we will prove the smoothness of the moduli space $\mathcal{M}^{\tilde{\mathbb{P}}}_{\vec{a}, k}=P/G$ where $P$ is the space of the ADHM data $\rho:=(a,q^{A},c,d)$ and $G$ is the symmetry group acting on this data as described in section \textbf{\ref{ADHM}}.
First let $\mathcal{E}$ and $\mathcal{E}'$ be two framed torsion-free sheaves with the same fixed Chern class, and let $\alpha:\mathcal{E}\longrightarrow\mathcal{E}'$ be a morphism preserving the framing up to a homothety, i.e., the diagram
\begin{equation}
\xymatrix@C-0.5pc@R-0.5pc{\mathcal{E}\ar[r]\ar[d]_{\alpha}&\mathcal{E}|_{l_{\infty}}\ar[r]^{\Phi}\ar[d]_{\alpha|_{l_{\infty}}}& \mathcal{O}|_{l_{\infty}}^{\oplus r}\ar[d]^{\lambda}\\
\mathcal{E}' \ar[r]&\mathcal{E}'|_{l_{\infty}}\ar[r]_{\Phi'}&\mathcal{O}|_{l_{\infty}}^{\oplus r}
}
\end{equation}
commutes. $\lambda\in\mathcal{O}^{\ast}|_{l_{\infty}}$ and $r$ is the rank of the sheaves $\mathcal{E}$ and $\mathcal{E}'$. Since $\Phi$ and $\Phi'$ are isomorphisms, one gets the following relation
\begin{equation}\label{relation}
\alpha|_{l_{\infty}}=\lambda\Phi'^{-1}\Phi.
\end{equation}
Since $Supp(|H|)$ is open dense in $\tilde{\mathbb{P}}$, and the fixed line $l_{\infty}$ is linearly equivalent to $H$, then the morphism $\alpha$ is completely determined by its restriction $\alpha|_{l_{\infty}}.$
Define $\Hom^{\Phi}(\mathcal{E},\mathcal{E}')$ to be the subgroup of $\Hom(\mathcal{E},\mathcal{E}')$ which contains the morphisms preserving the framing up to a homothety. $\Phi'^{-1}\Phi$ being a fixed element of $\End(\mathcal{O}|_{l_{\infty}}^{\oplus r})$ and $\lambda\in\mathcal{O}|_{l_{\infty}}$, the dimension of $\Hom^{\Phi}(\mathcal{E},\mathcal{E}')$ as a subspace of $\Hom(\mathcal{E},\mathcal{E}')$ is $1$.

Let us consider the universal monad on $\tilde{\mathbb{P}}\times P$ given by \eqref{univ-monad}; its cohomology is the $P$-flat family $\mathfrak{F}$. For every point $\rho\in P$, the fiber $\mathfrak{F}_{\rho}\cong\mathcal{E}(\rho)$ is a framed torsion-free sheaf. On $P\times P\times\tilde{\mathbb{P}}$ one has the following natural projections:
$$\xymatrix@R-0.5pc{P\times P\times\tilde{\mathbb{P}}\ar[d]_{pr_{12}}\ar@<1ex>[r]^{pr_{13}}\ar@<-1ex>[r]_{pr_{23}}& P\times\tilde{\mathbb{P}}\\
P\times P& }$$ where $pr_{12}$ projects on the first and the second factors, $pr_{13}$ projects on the first and the third factors and $pr_{23}$ projects on the second and the third factors.
Consider the sheaf $\mathcal{H}om(pr_{13}^{\ast}\mathfrak{F},pr_{23}^{\ast}\mathfrak{F})$ on $P\times P\times\tilde{\mathbb{P}}$.
The sheaves $pr_{23}^{\ast}\mathfrak{F}$ and $pr_{23}^{\ast}\mathfrak{F}$ are flat on $P\times P$, and then $\mathcal{H}om(pr_{13}^{\ast}\mathfrak{F},pr_{23}^{\ast}\mathfrak{F})$ is flat on $P\times P$.

Let us use the following notation: first we omit the pull-back symbols in order to avoid overloading text and formulas. Then we put $(\mathfrak{F}, \phi):=\xymatrix{\mathfrak{F}\ar[r]^\phi & \mathcal{O}^{\oplus r}_{l_{\infty}}}$ where the morphism $\phi$ is given by the triangle

$$\xymatrix@C-0.5pc@R-0.5pc{\mathfrak{F}\ar[r]^r \ar[dr]_\phi & \mathfrak{F}|_{l_{\infty}}\ar[d]^\Phi \\ & \mathcal{O}^{\oplus r}_{l_{\infty}} }$$ and define
\begin{displaymath}
\Hom((\mathfrak{F}, \phi),(\mathfrak{F}', \psi)):=\Bigg\{\begin{array}{ll}\alpha:\mathfrak{F}\longrightarrow\mathfrak{F}'& \\ \lambda\in\mathcal{O}_{l_{\infty}}& \end{array} {\Large\Bigg{|}} \small{\textrm{ such that }} \quad\begin{array}{l} \\ \xymatrix@R-0.5pc@C-0.5pc{\ar@{}[dr] |{\circlearrowright}\mathfrak{F}\ar[r]^\alpha \ar[d]^\phi & \mathfrak{F}'\ar[d]^\psi \\ \mathcal{O}^{\oplus r}_{l_{\infty}}\ar[r]^\lambda & \mathcal{O}^{\oplus r}_{l_{\infty}}} \\ \\ \end{array}\Bigg\}
\end{displaymath}

Now define the pre-sheaf $\mathcal{H}om((\mathfrak{F}, \phi),(\mathfrak{F}', \psi))$ given by
\begin{equation}\label{sheaf}
U\longrightarrow \Hom((\mathfrak{F}|_{U}, \phi|_{U}),(\mathfrak{F}'|_{U}, \psi|_{U})),
\end{equation}
i.e., to every open subset $U$ such that $U\cap l_{\infty}\neq\emptyset$ (for example $\tilde{\mathbb{P}}\backslash E_{i},$ where $E_{i}$ is one of the exceptional divisors for example) we associate the diagram
\begin{equation}
\xymatrix@C-0.5pc@R-0.5pc{\ar@{}[dr] |{\circlearrowright}\mathfrak{F}|_{U}\ar[r]^{\alpha|_{U}} \ar[d]_{\phi|_{U}} & \mathfrak{F}'|_{U}\ar[d]^{\psi|_{U}} \\ \mathcal{O}^{\oplus r}_{U\cap l_{\infty}}\ar[r]^\lambda & \mathcal{O}^{\oplus r}_{U\cap l_{\infty}}.}
\end{equation}
and to every open $U$ such that $U\cap l_{\infty}=\emptyset$ (for instance $\tilde{\mathbb{P}}\backslash l_{\infty}$) we just associate $\Hom(\mathfrak{F}|_{U},\mathfrak{F}'|_{U}).$ Moreover this is a sheaf; the sheaf axioms are inherited from the sheaf properties of $\mathfrak{F}$, $\mathfrak{F}'$ and the commutation of the square diagrams.
\begin{pr}
The sheaf $\mathcal{H}om((\mathfrak{F}, \phi),(\mathfrak{F}', \psi))$ is flat over $P\times P$
\end{pr}

\begin {proof}
Since the question is local on $P\times P,$ one can consider an open affine $W=\spec A\subset P\times P$ and work with $A$-modules, where $A$ is a commutative noetherian ring. Then one can show that $\mathcal{H}om((\mathfrak{F}, \phi),(\mathfrak{F}', \psi))$ is $A$-flat: let
\begin{equation}\label{modules}
0\longrightarrow M'\longrightarrow M\longrightarrow M''\longrightarrow0
\end{equation}
be an exact sequence of $A$-modules, and let $\mathfrak{F}$ be an $A$-flat family of framed torsion-free sheaves on $\tilde{\mathbb{P}}$. One has the short exact sequence $$0\longrightarrow\mathfrak{F}\otimes_{A}M'\longrightarrow\mathfrak{F}\otimes_{A}M\longrightarrow\mathfrak{F}\otimes_{A}M''\longrightarrow0.$$ The restriction of $\mathfrak{F}$ to any open $U$ is also an $A$-flat module, then one has the sequence $$0\longrightarrow\mathfrak{F}|_{U}\otimes_{A}M'\longrightarrow\mathfrak{F}|_{U}\otimes_{A}M\longrightarrow \mathfrak{F}|_{U}\otimes_{A}M''\longrightarrow0$$
The same situation is true for any other such family $\mathfrak{F}'$, and if we consider any morphism $\alpha|_{U}:\mathfrak{F}|_{U}\longrightarrow\mathfrak{F}'|_{U}$
in $\Hom((\mathfrak{F}|_{U}, \phi|_{U}),(\mathfrak{F}'|_{U}, \psi|_{U}))$, we get the following diagram:
{\scriptsize
\begin{displaymath}
\xymatrix@C-1.6pc@R-0.9pc{&0\ar[rr]&&\mathfrak{F}|_{U}\otimes_{A}M'\ar[rr]\ar@{->}'[d][ddd]|{\alpha|_{U}\otimes I_{M'}}\ar[dl]|{\phi|_{U}\otimes_{A}I_{M'}}&&\mathfrak{F}|_{U}\otimes_{A}M\ar[rr]\ar@{->}'[d][ddd]|{\alpha|_{U}\otimes I_{M}}\ar[dl]|{\phi|_{U}\otimes_{A}I_{M}}&&
\mathfrak{F}|_{U}\otimes_{A}M''\ar[rr]\ar@{->}'[d][ddd]|{\alpha|_{U}\otimes I_{M''}}\ar[dl]|{\phi|_{U}\otimes_{A}I_{M''}}&&0\\
0\ar[rr]&&\mathcal{O}^{\oplus r}_{U\cap l_{\infty}}\otimes_{A}M'\ar[rr]\ar@{-}[d]&&\mathcal{O}^{\oplus r}_{U\cap l_{\infty}}\otimes_{A}M\ar[rr]\ar@{-}[d]&&\mathcal{O}^{\oplus r}_{U\cap l_{\infty}}\otimes_{A}M''\ar[rr]\ar@{-}[d]&&0&\\
&&\lambda\otimes I_{M'}\ar[dd]&&\lambda\otimes I_{M}\ar[dd]&&\lambda\otimes I_{M''}\ar[dd]&&&\\
&0\ar@{->}'[r][rr]&&\mathfrak{F}'|_{U}\otimes_{A}M'\ar@{->}'[r][rr]\ar[dl]|{\psi|_{U}\otimes_{A}I_{M'}}&
&\mathfrak{F}'|_{U}\otimes_{A}M\ar@{->}'[r][rr]\ar[dl]|{\psi|_{U}\otimes_{A}I_{M}}&
&\mathfrak{F}'|_{U}\otimes_{A}M''\ar[rr]\ar[dl]|{\psi|_{U}\otimes_{A}I_{M''}}&&0\\
0\ar[rr]&&\mathcal{O}^{\oplus r}_{U\cap l_{\infty}}\otimes_{A}M'\ar[rr]&&\mathcal{O}^{\oplus r}_{U\cap l_{\infty}}\otimes_{A}M\ar[rr]&&\mathcal{O}^{\oplus r}_{U\cap l_{\infty}}\otimes_{A}M''\ar[rr]&&0&\\
}
\end{displaymath}
}

where $I_{M},$ $I_{M'}$ and $I_{M''}$ are the identity morphisms on $M$, $M'$ and $M''$ respectively.   
%Equivalently this means that the sequence
%{\footnotesize
%$$\xymatrix@1{0\ar[r]&\Hom((\mathfrak{F}|_{U}\otimes_{A} M',\phi|_{U}\otimes_{A}I_{M'}),(\mathfrak{F}'|_{U}\otimes_{A} %M',\psi|_{U}\otimes_{A}I_{M'}))}\qquad\qquad\qquad\qquad\qquad\qquad\qquad\qquad$$
%$$\quad\qquad\qquad\xymatrix@1{\ar[r]&\Hom((\mathfrak{F}|_{U}\otimes_{A} M,\phi|_{U}\otimes_{A}I_{M}),(\mathfrak{F}'|_{U}\otimes_{A} %M,\psi|_{U}\otimes_{A}I_{M}))&}$$
%$$\qquad\qquad\qquad\qquad\qquad\qquad\xymatrix@1{\ar[r]&\Hom((\mathfrak{F}|_{U}\otimes_{A} %M'',\phi|_{U}\otimes_{A}I_{M''}),(\mathfrak{F}'|_{U}\otimes_{A} M'',\psi|_{U}\otimes_{A}I_{M''}))\ar[r]&0}$$}
%is exact.

\bigskip

On the other hand, if we twist the sequence \eqref{modules} by the $A-$module $\Hom((\mathfrak{F}|_{U},\phi|_{U}),(\mathfrak{F}'|_{U},\psi|_{U}))$, then we get the sequence {\small

$$\Hom((\mathfrak{F}|_{U},\phi|_{U}),(\mathfrak{F}'|_{U},\psi|_{U}))\otimes_{A}M'\stackrel{\Xi}{
\longrightarrow} \Hom((\mathfrak{F}|_{U},\phi|_{U}),(\mathfrak{F}'|_{U},\psi|_{U}))\otimes_{A}M \qquad\qquad$$ $$\qquad\qquad\qquad\longrightarrow \Hom((\mathfrak{F}|_{U},\phi|_{U}),(\mathfrak{F}'|_{U},\psi|_{U}))\otimes_{A}M''\longrightarrow0.$$ }

\vspace{0.2cm}

\underline{{\bf Claim:}} The map $\Xi,$ in the sequence above, is injective.

\vspace{0.2cm}

If we denote the first map in \eqref{modules} by $\theta$, then $\Xi:\Sigma_{i}\alpha_{i}\otimes_{A}m_{i}\longrightarrow\Sigma_{i}\alpha_{i}\otimes_{A}\theta(m_{i})$ for all $\alpha_{i}$'s as in the three-dimensional commutative diagram above. By the commutativity of the following diagram
\begin{displaymath}
\xymatrix@C-0.6pc@R-0.6pc{0\ar[r]&\mathfrak{F}|_{U}\otimes_{A}M'\ar[r]\ar[d]_{\alpha_{i}\otimes_{A}I_{M'}}&\mathfrak{F}|_{U}\otimes_{A}M \ar[r]\ar[d]^{\alpha_{i}\otimes_{A}I_{M}}&\mathfrak{F}|_{U}\otimes_{A}M''\ar[r]\ar[d]&0 \\
0\ar[r]&\mathfrak{F}'|_{U}\otimes_{A}M'\ar[r]^{\tilde{\theta}}&\mathfrak{F}'|_{U}\otimes_{A}M\ar[r]&\mathfrak{F}'|_{U}\otimes_{A}M''\ar[r]&0
}
\end{displaymath}
it follows that if $\alpha_{i}\otimes_{A}I_{M}=0,$ then $\tilde{\theta}\circ\alpha_{i}\otimes_{A}I_{M'}=0.$ Moreover $\tilde{\theta}$ is injective, hence $\alpha_{i}\otimes_{A}I_{M'}=0$. Thus $\Xi$ is injective.
After gluing sections of the involved sheaves, one gets an exact sequence

{\small
$$\xymatrix@1{0\ar[r]&\mathcal{H}om((\mathfrak{F},\phi),(\mathfrak{F}',\psi))\otimes_{A}M'\ar[r]&
\mathcal{H}om((\mathfrak{F},\phi),(\mathfrak{F}',\psi))\otimes_{A}M&}$$
$$\xymatrix@1{\ar[r]&\mathcal{H}om((\mathfrak{F},\phi),(\mathfrak{F}',\psi))\otimes_{A}M''\ar[r]&0}$$}
Hence the sheaf $\mathcal{H}om((\mathfrak{F},\phi),(\mathfrak{F}',\psi))$ is flat over $P\times P$.

\end{proof}

\bigskip

Let us remark that since the group $G$ acts freely on the non-singular space $P$, with trivial stabilizer for any point $\rho\in P$, then the quotient is smooth if the graph of the group action is closed, that is, the image $\Gamma:=Im\gamma$ of the morphism $$\gamma:G\times P\longrightarrow P\times P$$ is closed \cite[Chapter \textbf{I}, Theorem \textbf{1}]{Sesha}. In our case the pair $(\rho,\sigma)$ is in the graph $\Gamma$ if and only if $dim \Hom^{\Phi}(\mathcal{E}(\rho),\mathcal{E}(\sigma))=1$, in other words
\begin{equation}
\Gamma =\left\{ (\rho,\sigma)\in P\times P \textbf{/}
h^{0}(pr_{12}^{-1}(\rho,\sigma),\mathcal{H}om^{\Phi}(pr_{13}^{\ast}\mathfrak{F}(\rho),pr_{23}^{\ast}\mathfrak{F}(\sigma)))>0 \right\}
\end{equation}
Since the sheaf $\mathcal{H}om^{\Phi}(pr_{13}^{\ast}\mathfrak{F},pr_{23}^{\ast}\mathfrak{F})$ is flat on $P\times P$, then it follows from the semi-continuity theorem that $\Gamma$ is a closed subscheme. Hence the quotient $P/G$ is smooth.

\bigskip

\section{Fineness of the Moduli space}

In this section we shall prove that $\mathcal{M}^{\tilde{\mathbb{P}}}_{\vec{a}, k}$ is a fine moduli space by using the monads obtained in section \textbf{\ref{unversal}}. We start by showing that it is a coarse moduli space. The construction of the following natural transformation $$\Phi: \mathfrak{M}^{\tilde{\mathbb{P}}}_{\vec{a}, k}(\bullet)\longrightarrow \Hom(\bullet,\mathcal{M}^{\tilde{\mathbb{P}}}_{\vec{a}, k});$$ goes as the following: for any classifying scheme $S$,  $\xi\in \mathfrak{M}^{\tilde{\mathbb{P}}}_{\vec{a}, k}(S)$, and a family $\mathcal{F}$ of torsion-free sheaves on $\tilde{\mathbb{P}}$, with Chern character $r+\Sigma_{i=1}^{n}a_{i}E_{i}-(k-\frac{|\vec{a}|^{2}}{2})\omega$, classified by $S$, one has a monad given as in \eqref{S-monad} which is canonically associated to $\mathcal{F}$. If we consider an open covering $\{S_{j}\}_{j\in J}$ of $S$, then on every open affine $S_{j}$ the restriction $\mathcal{M}|_{S_{j}}$ is isomorphic to a monad of the form:
$$\xymatrix@C-0.8pc{\mathbb{M}(\alpha_{j},\beta_{j}):& \oplus_{i=0}^{n}\mathcal{O}_{\tilde{\mathbb{P}}}(-1,E_{i})\boxtimes K_{i}\otimes\mathcal{O}_{S_{j}}\ar[r]^{\quad\quad\quad\alpha_{j}}&  \mathcal{O}_{\tilde{\mathbb{P}}}\boxtimes W\otimes\mathcal{O}_{S_{j}}\ar[r]^{\beta_{j}\quad\quad\quad}& \oplus_{i=0}^{n}\mathcal{O}_{\tilde{\mathbb{P}}}(1,-E_{i})\boxtimes L_{i}\otimes\mathcal{O}_{S_{j}}}$$
where $$\alpha_{j}:S_{j}\longrightarrow \mathbb{H}, \quad\quad\quad\beta_{j}:S_{j}\longrightarrow \mathbb{F}.$$ The spaces $\mathbb{H}$, $\mathbb{F}$ are defined by
$$\left\{\begin{array}{l}\mathbb{H}=\oplus_{i=0}^{n}\Hom(V_{i}, \Hom(K_{i},W))\\ \mathbb{F}=\oplus_{i=0}^{n}\Hom(V_{i}, \Hom(W,L_{i}))\end{array}\right.$$ and $V_{i}$ is such that $V_{i}^{\ast}=\h^{0}(\tilde{\mathbb{P}},\mathcal{O}(1,-E_{i})).$ From the monad condition, $\beta_{i}\circ\alpha_{i}=0,$ we have regular maps $f_{j}=(\alpha_{j},\beta_{j}):S_{j}\longrightarrow P.$ By construction these morphisms satisfy $$f_{i}(s)\sim_{G} f_{j}(s)$$ for $s$ in the intersection $S_{i}\cap S_{j}$, where $G$ is the group defined in \eqref{group-form} with the action given by \eqref{P and G}. The maps $f_{j}$ glue to form a global morphism $$f:S\longrightarrow \mathcal{M}^{\tilde{\mathbb{P}}}_{\vec{a}, k}.$$
This defines the desired natural transformation:
\begin{equation}\label{natural-transform}
\xymatrix@R-1pc{\Phi:& \mathfrak{M}^{\tilde{\mathbb{P}}}_{\vec{a}, k}(\bullet)\ar[r]& \Hom(\bullet,\mathcal{M}^{\tilde{\mathbb{P}}}_{\vec{a}, k})}
\end{equation}
by the association $\xymatrix{\xi\ar[r]&\Phi(\xi):=f}$ for every class $\xi=[\mathcal{F}]$ on a given scheme $S.$ Obviously $f$ depends only on the class $\xi=[\mathcal{F}].$ By using the monad, on $\tilde{\mathbb{P}},$ associated to a closed point $s\in S$ (obtained by restricting the monad $\mathcal{M}$ to the point $s\in S$) one deduces that $\Phi: \mathfrak{M}^{\tilde{\mathbb{P}}}_{\vec{a}, k}(\spec k(s))\longrightarrow \Hom(\spec k(s),\mathcal{M}^{\tilde{\mathbb{P}}}_{\vec{a}, k})$ is a bijection.
Now let $\mathcal{R}$ be another parameterizing scheme such that there is a natural transformation
$$\Psi: \mathfrak{M}^{\tilde{\mathbb{P}}}_{\vec{a}, k}(\bullet)\longrightarrow \Hom(\bullet,\mathcal{R}).$$

If $\mathfrak{F}$ is a universal family on $\tilde{\mathbb{P}}\times P$ parameterized by $P$ such that $\Phi(\mathfrak{F})=\pi:P\longrightarrow \mathcal{M}^{\tilde{\mathbb{P}}}_{\vec{a}, k}$, then, for the natural transformation $\Psi$, we have $\Psi(\mathfrak{F}):P\longrightarrow \mathcal{R}$.

\begin{pr}
$\Psi(\mathfrak{F})$ is constant along the fibers of the projection $\pi:P\longrightarrow \mathcal{M}^{\tilde{\mathbb{P}}}_{\vec{a}, k}.$
\end{pr}
\begin{proof}
Let $p=\spec k(p)\in P$ and let $\rho_{1}$, $\rho_{2}\in \Hom(p,P)$ such that $\pi(\rho_{1})=\pi(\rho_{2})$; if we consider the pull-back
\begin{equation}
\xymatrix@C-0.4pc@R-0.5pc{ &\rho_{i}^{\ast}\mathfrak{F}\ar[r]\ar[d]& \mathfrak{F}\ar[d]\\
&\tilde{\mathbb{P}}\times p\ar[r]^{(Id_{\tilde{\mathbb{P}}}\times \rho_{i})}& \tilde{\mathbb{P}}\times P
}
\end{equation}
then $\Phi(\rho_{1}^{\ast}\mathfrak{F})=\Phi(\mathfrak{F})(\rho_{1})$, and by definition $\Phi(\mathfrak{F})(\rho_{1})=\pi(\rho_{1})$. By assumption, we have $\pi(\rho_{1})=\pi(\rho_{2})=\Phi(\mathfrak{F})(\rho_{2})=\Phi(\rho_{2}^{\ast}\mathfrak{F})$, and since the natural transformation $\Phi$ is a bijection for every closed point, it follows that $$\rho_{1}^{\ast}\mathfrak{F}=\rho_{2}^{\ast}\mathfrak{F}$$
On the other hand we have $$\Psi(\mathfrak{F})(\rho_{1})=\Psi(\rho_{1}^{\ast}\mathfrak{F})=\Psi(\rho_{2}^{\ast}\mathfrak{F})=\Psi(\mathfrak{F})(\rho_{2}).$$ Thus $\Psi(\mathfrak{F})$ is constant along the fibers of $\pi:P\longrightarrow \mathcal{M}^{\tilde{\mathbb{P}}}_{\vec{a}, k}$
\end{proof}
The projection $\pi:P\longrightarrow \mathcal{M}^{\tilde{\mathbb{P}}}_{\vec{a}, k}$ locally has sections, so one can construct local mappings $\bar{\phi}:\mathcal{M}^{\tilde{\mathbb{P}}}_{\vec{a}, k}\longrightarrow \mathcal{R}$, but since $\Psi(\mathfrak{F})$ is constant along the fibers of $\pi$, then the map $\bar{\phi}$ can be lifted to a global map $\phi$ such that the following diagram commutes:
\begin{equation}
\xymatrix@C-0.5pc@R-0.5pc{P \ar[r]^{\Psi(\mathfrak{F})}\ar[d]_{\Phi(\mathfrak{F})=\pi}& \mathcal{R}\\
\mathcal{M}^{\tilde{\mathbb{P}}}_{\vec{a}, k}\ar[ru]_{\phi}&
}
\end{equation}

Now for any parameterizing scheme $S$ and a family $\xi$ on $\tilde{\mathbb{P}}\times S$, one has $\Psi(\xi):S\longrightarrow\mathcal{R}$. Let $\{S_{i}\}_{i\in I}$ be an open cover of $S$. The diagram
\begin{displaymath}
\xymatrix@C-0.5pc@R-0.5pc{S \ar[dr]\ar@/^/@{.>}[drr]^{\Psi(\xi)}&& \\
&P\ar[r]& \mathcal{R}\\
S_{i}\ar@{^{(}->}[uu]^{g_{i}}\ar@{-->}[ur]\ar[urr]&&
}
\end{displaymath}
commutes, and we have $\xi|_{S_{i}}=g_{i}^{\ast}(\mathfrak{F})$. Consequently
\begin{align}
\Psi(\xi)|_{S_{i}}&=\Psi(\xi|_{S_{i}})\notag \\
&=\Psi(g_{i}^{\ast}(\mathfrak{F})) \notag \\
&=g_{i}^{\ast}\Psi(\mathfrak{F}) \notag
\end{align}
On the other hand $\Psi(\mathfrak{F})=\phi\circ\Phi(\mathfrak{F})$, hence
\begin{align}
\Psi(\xi)|_{S_{i}}&=g_{i}^{\ast}(\phi\circ\Phi(\mathfrak{F}))\notag \\
&=\phi\circ\Phi(g_{i}^{\ast}(\mathfrak{F})) \notag \\
&=\phi\circ\Phi(\xi|_{S_{i}}) \notag \\
\Psi(\xi)|_{S_{i}}&=[\phi\circ\Phi(\xi)]|_{S_{i}} \notag
\end{align}

These maps glue together to form a global map $$\Psi(\xi)=\phi\circ\Phi(\xi)$$ on $S$. Since $\mathcal{M}^{\tilde{\mathbb{P}}}_{\vec{a}, k}$ is reduced (this follows from the smoothness), then the map $\phi$ is uniquely determined. By this we showed the following:

\begin{thm}
The scheme $\mathcal{M}^{\tilde{\mathbb{P}}}_{\vec{a}, k}$ is a coarse moduli space.
\end{thm}

\bigskip

The final step is to show that $\mathcal{M}^{\tilde{\mathbb{P}}}_{\vec{a}, k}$ is fine. To do this we shall descend the universal monadic description on $\tilde{\mathbb{P}}\times P$ to a well behaved monadic description on $\tilde{\mathbb{P}}\times\mathcal{M}^{\tilde{\mathbb{P}}}_{\vec{a}, k}$. This can be realized, in our case, because the space $\tilde{\mathbb{P}}\times P$ is a $G$-space: indeed there is a natural action
\begin{equation*}
\xymatrix@R-1.5pc{G\times\tilde{\mathbb{P}}\times P \ar[r]&\tilde{\mathbb{P}}\times P \\
(g,(x,\mathcal{C}))\ar[r]&(x,g\cdot\mathcal{C})}
\end{equation*}
which induces a $G$-action on the universal monad $\mathbb{M}$, in \eqref{univ-monad}, and which descends to an action on its cohomology $\mathfrak{F}.$ Since the action is free and the isotropy subgroup is trivial at all points, we have a well defined family $\mathfrak{F}/G\longrightarrow\tilde{\mathbb{P}}\times P/G$. We put $\mathfrak{U}:=\mathfrak{F}/G$ which is a canonical family $$\mathfrak{U}\longrightarrow\tilde{\mathbb{P}}\times\mathcal{M}^{\tilde{\mathbb{P}}}_{\vec{a}, k}$$ parameterized by $\mathcal{M}^{\tilde{\mathbb{P}}}_{\vec{a}, k}$.
\begin{pr}
For any noetherian  scheme $S$ of finite type, the mapping
$$\begin{array}{cccc}\Hom(S,\mathcal{M}^{\tilde{\mathbb{P}}}_{\vec{a}, k}) & \longrightarrow & \mathfrak{M}^{\tilde{\mathbb{P}}}_{\vec{a}, k}(S)\\
\phi & \longrightarrow & \phi^{\ast}[\mathfrak{U}]=[(Id_{\tilde{\mathbb{P}}}\times\phi)^{\ast}\mathfrak{U}]\end{array}$$ is bijective.
\end{pr}

\begin{proof}

\underline{Injectivity}:

\bigskip

Let $\phi_{1}, \phi_{2}: S\longrightarrow\mathcal{M}^{\tilde{\mathbb{P}}}_{\vec{a}, k}$ such that $(Id_{\tilde{\mathbb{P}}}\times\phi_{1})^{\ast}\mathfrak{U}\cong(Id_{\tilde{\mathbb{P}}}\times\phi_{2})^{\ast}\mathfrak{U}$
then for every point $s\in S$, one has $\mathfrak{U}(\phi_{1}(s))=\mathfrak{U}(\phi_{2}(s))$. Since the torsion-free sheaf $\mathfrak{U}(\phi_{i}(s))$ is the one given by the ADHM data associated to the point $\phi_{i}(s)\in \mathcal{M}^{\tilde{\mathbb{P}}}_{\vec{a}, k}$, then $\phi_{1}(s)=\phi_{2}(s)$ for every point $s\in S$, thus $\phi_{1}=\phi_{2}.$

\bigskip

\underline{Surjectivity}:

\bigskip

Given a family $\mathcal{F}$ parameterized by $S,$ one has the morphism $\phi=\Phi(\mathcal{F})$ given by the natural transformation \eqref{natural-transform}. Then $\mathcal{F}$ is the pull-back of the family $\mathfrak{U}$ parameterized by $\mathcal{M}^{\tilde{\mathbb{P}}}_{\vec{a}, k}$.
\end{proof}

This finishes the proof of the following

\begin{thm}
The scheme $\mathcal{M}^{\tilde{\mathbb{P}}}_{\vec{a}, k}$ is a fine moduli space for the moduli functor $\mathfrak{M}^{\tilde{\mathbb{P}}}_{\vec{a}, k}.$
\end{thm}

\end{document}